\documentclass[
a4paper 
,11pt 
,leqno 
,english 
]{scrartcl} 

\usepackage{my-preamble-article} 
\usepackage{my-macros} 
\usepackage{my-specific-macros} 
\usepackage{my-counters-related-preprints} 
\title{Existence of a pulled or pushed travelling front invading a critical point for parabolic gradient systems}
\author{\RamonsName{} and \myName}
\begin{document}
\hypersetup{pageanchor=false} 
\maketitle
\nnfootnote{%
}
\begin{abstract}
For nonlinear parabolic gradient systems of the form
\[
u_t = -\nabla V(u) + u_{xx}
\,,
\]
where the spatial domain is the whole real line, the state variable $u$ is multidimensional, and the potential function $V$ is coercive at infinity, the following result is proved: for every critical point of $V$ which is not a global minimum point, there exists a travelling front, either pushed or pulled, invading this critical point at a speed which is not smaller than its linear spreading speed. By contrast with previous existence results of the same kind, no further assumption is made (neither that the invaded critical point is a non-degenerate local minimum point, nor other assumptions ensuring pushed invasion).
\end{abstract}
\thispagestyle{empty}
\pagestyle{empty}
\newpage%
\tableofcontents
\newpage
\hypersetup{pageanchor=true} 
\pagestyle{plain}
\setcounter{page}{1}
\section{Introduction}
\subsection{System}
Let us consider the nonlinear parabolic system
\begin{equation}
\label{parabolic_system}
u_t=-\nabla V (u) + u_{xx}
\,,
\end{equation}
where the time variable $t$ and the space variable $x$ are real, the spatial domain is the whole real line, the function $(x,t)\mapsto u(x,t)$ takes its values in $\rr^d$ with $d$ a positive integer, and the nonlinearity is the gradient of a \emph{potential} function $V:\rr^d\to\rr$, of class $\ccc^2$, and \emph{coercive at infinity} in the following sense:
\begin{gather} 
\tag{$\text{H}_\text{coerc}$}
\lim_{R\to+\infty}\quad  \inf_{\abs{u}\ge R}\ \frac{u\cdot \nabla V(u)}{\abs{u}^2} >0
\,.
\label{hyp_coerc}
\end{gather}
Let us assume that
\begin{gather} 
\tag{$\text{H}_{\text{crit}}$}
\nabla V(0_{\rr^d})=0
\quad\text{and}\quad
V(0_{\rr^d})=0
\quad\text{and}\quad
\min_{u\in\rr^d}V(u)<0
\,.
\label{hyp_crit_point}
\end{gather}
In other words, the origin $0_{\rr^d}$ of $\rr^d$ is assumed to be a critical point which is \emph{not} a global minimum of $V$, and $V$ is normalized so that it takes the value $0$ at $0_{\rr^d}$.
\subsection{Travelling waves/fronts}
Let $c$ be a \emph{positive} real quantity. A \emph{wave travelling at the speed $c$} for system \cref{parabolic_system} is a function of the form $(x,t)\mapsto \phi(x-ct)$, where $\phi$ is a solution (with values in $\rr^d$) of the second order differential system
\begin{equation}
\label{syst_trav_front_order_2}
\phi''=-c\phi'+\nabla V(\phi) 
\,,
\end{equation}
which is equivalent (denoting $\phi'$ by $\varphi$) to the first order differential system 
\begin{equation}
\label{syst_trav_front_order_1}
\begin{pmatrix} \phi' \\ \varphi' \end{pmatrix} = \begin{pmatrix}\varphi \\ - c \varphi + \nabla V(\phi)  \end{pmatrix}
\,.
\end{equation}
The function $\phi$ is called the \emph{profile} of the travelling wave. 
\begin{definition}[travelling front invading the critical point $0_{\rr^d}$]
\label{def:travelling_front_invading_0_speed_c}
Let us call \emph{travelling front invading $0_{\rr^d}$ at the speed $c$} a wave travelling at the speed $c$ such that its profile $\xi\mapsto\phi(\xi)$ is \emph{nonconstant}, bounded (therefore globally defined), and satisfies the limit
\[
\phi(\xi)\to0_{\rr^d} 
\quad\text{as}\quad
\xi\to+\infty
\,.
\]
\end{definition}
Let $\SigmaCrit(V)$ denote the set of critical points of $V$; with symbols, 
\[
\SigmaCrit(V) = \{u\in\rr^d: \nabla V(u) = 0\}
\,.
\]
It follows from \cite[Lemma~9]{Risler_globCVTravFronts_2008} and \cite[conclusion~6 of Lemma~7.1]{Risler_globalBehaviour_2016} (among other possible references) that, if $\phi$ is the profile of a travelling front invading $0_{\rr^d}$ at the speed $c$ in the sense of \cref{def:travelling_front_invading_0_speed_c}, then the quantity 
\begin{equation}
\label{uniform_bound_profile_front}
\sup_{\xi\in\rr}\abs{\phi(\xi)}
\end{equation}
is bounded from above by a quantity depending only on $V$, and there exists a negative quantity $V_{\-\infty}$ such that the following limit holds:
\[
\dist\Bigl(\phi(\xi),\SigmaCrit(V)\cap V^{-1}\bigl(\{V_{-\infty}\}\bigr)\Bigr)\to 0
\quad\text{as}\quad
\xi\to -\infty
\,.
\]
\begin{remark}
Voluntarily, this definition of a travelling front slightly differs from the usual one, since it does not require $\phi(\xi)$ to approach a single critical point of $V$ as $\xi$ goes to $-\infty$. In most cases (for instance if the critical points of $V$ are isolated --- which is true for a generic potential $V$ --- or if $V$ is analytic, \cite{HarauxJendoubi_CVSol2ndOrderGradAnalytic_1998}), this actually makes no difference: for every travelling front in the sense of \cref{def:travelling_front_invading_0_speed_c}, the profile \emph{does} necessarily approach a single critical point of $V$ at the left end of space. However, for certain particular potentials, it may actually happen that the limit set, as $\xi$ goes to $-\infty$, of the profile $\phi(\xi)$ of such a travelling front is \emph{not} reduced to a single critical point; an example (for a potential $V$ of class $C^\infty$) is provided in \cite[Corollary~5.2]{JendoubiPolacik_nonStabSolSemilinHypElliptEquDamp_2003}. 
\end{remark}
\subsection{Linearization at the invaded critical point}
The linearization of the (equivalent) differential systems \cref{syst_trav_front_order_2,syst_trav_front_order_1} at the point $(0_{\rr^{2d}})$ reads:
\begin{equation}
\label{syst_trav_front_order_1_2_linearized}
\phi'' = -c\phi' + D^2V(0_{\rr^d})\cdot\phi
\quad\text{and}\quad
\frac{d}{d\xi}\begin{pmatrix}\phi \\ \varphi\end{pmatrix} = \begin{pmatrix} 0_d & I_d \\ D^2V(0_{\rr^d}) & -c I_d\end{pmatrix}\cdot\begin{pmatrix}\phi \\ \varphi\end{pmatrix}
\,,
\end{equation}
where $0_d$ and $I_d$ denote the $d\times d$ zero-matrix and identity matrix, respectively. Let $\mu_1,\dots,\mu_d$ denote the eigenvalues of $D^2V(0_{\rr^d})$, counted with algebraic multiplicity. Without loss of generality, it may be assumed that
\begin{equation}
\label{sort_mu_j}
\mu_1\le\dots\le\mu_d
\,. 
\end{equation}
The eigenvalues of system \cref{syst_trav_front_order_1_2_linearized} read
\[
-\frac{c}{2}\pm\sqrt{\frac{c^2}{4}+\mu_j} \quad\text{if}\quad -\frac{c^2}{4} \le \mu_j
\quad\text{and}\quad
-\frac{c}{2}\pm i\sqrt{-\frac{c^2}{4}-\mu_j} \quad\text{if}\quad \mu_j \le -\frac{c^2}{4}\,,
\quad 1\le j\le d
\,.
\]
\begin{definition}[maximal linear invasion speed]
\label{def:max_linear_invasion_speed}
Let us call \emph{maximal linear invasion speed} (associated with the critical point $0_{\rr^d}$) the (nonnegative) quantity $\cLinMax$ defined as
\[
\cLinMax = \left\{
\begin{aligned}
2\sqrt{-\mu_1} = 2\sqrt{\abs{\mu_1}} \quad&\text{if}\quad \mu_1 < 0\,, \\
0\quad&\text{if}\quad \mu_1 \ge 0
\,.
\end{aligned}
\right.
\]
\end{definition}
\begin{remark}
This quantity $\cLinMax$ is usually called \emph{linear spreading speed} in the literature, see for instance \cite{VanSaarloos_frontPropagationUnstableStates_2003}, and is referred as such in the abstract of this article. In the following, only the denomination \emph{maximal linear invasion speed} will be used, except in the statement of the main result (\cref{thm:main}) where the shorter denomination \emph{linear spreading speed} is chosen. 
\end{remark}
\subsection{Pushed/pulled travelling fronts}
Let us keep the previous notation, and let $\phi$ denote the profile of a front invading the critical point $0_{\rr^d}$ at the speed $c$. Since $\phi(\xi)$ goes to $0_{\rr^d}$ as $\xi$ goes to $+\infty$, there must exist some nonnegative quantity $\lambda$, equal to either $c/2$ or to the opposite of one of the real eigenvalues of the linearized systems \cref{syst_trav_front_order_1_2_linearized}, such that
\begin{equation}
\label{def_steepness}
\frac{\ln\abs{\phi(\xi)}}{\xi}\to -\lambda
\quad\text{as}\quad 
\xi\to+\infty
\,.
\end{equation}
\begin{definition}[steepness of a travelling front invading $0_{\rr^d}$]
\label{def:steepness}
Let us call \emph{steepness} of the front under consideration the (nonnegative) quantity $\lambda$ defined by the limit \cref{def_steepness}. 
\end{definition}
\begin{definition}[pushed/pulled travelling wave/front invading $0_{\rr^d}$]
\label{def:pushed_pulled_travelling_front}
A travelling front invading the critical point $0_{\rr^d}$ at some positive speed $c$ is said to be:
\begin{itemize}
\item \emph{pushed} if its steepness $\lambda$ satisfies the inequality
\[
\frac{c}{2}<\lambda
\,,
\]
or equivalently if the following limit holds for its profile $\phi$:
\[
\phi(\xi)  = o\bigl(e^{-\frac{1}{2}c\xi}\bigr)
\quad\text{as}\quad
\xi\to+\infty
\,;
\]
\item \emph{pulled} if its speed $c$ equals the maximal linear invasion speed $\cLinMax$ and its steepness $\lambda$ equals $\cLinMax/2$. 
\end{itemize}
\end{definition}
\begin{remark}
In \cref{def:pushed_pulled_travelling_front} above, the qualifier \emph{pulled} is applied in a broad sense: no distinction is made between the two subclasses of pulled fronts, respectively defined by the refined asymptotics: 
\[
\abs{\phi(\xi)} \underset{\xi\to+\infty}{\sim} K \xi e^{-\frac{1}{2}\cLinMax\xi}
\quad\text{and}\quad
\abs{\phi(\xi)} \underset{\xi\to+\infty}{\sim} K e^{-\frac{1}{2}\cLinMax\xi}
\,,
\]
for some positive quantity $K$. The first (generic) subclass contains the pulled fronts in a more restricted sense, whereas the second (non-generic, codimension one) subclass contains fronts that are at the transition between pulled and pushed fronts, \cite{MontieHolzerScheel_pushedToPulledFrontTransitions_2022}; those are sometimes qualified as \emph{pushmi-pullyu} or \emph{variational pulled}, depending on authors and context, \cite{AnHendersonRyzhik_quantitativeSteepnessSemiKPPPushmiPullyu_2023,Muratov_globVarStructPropagation_2004}. 
\end{remark}
\subsection{Main result}
The following statement, which is the main result of this paper, calls upon:
\begin{itemize}
\item \cref{def:max_linear_invasion_speed} of the linear spreading speed (called ``maximal linear invasion speed'' everywhere else in the paper), 
\item and \cref{def:pushed_pulled_travelling_front} of pushed/pulled travelling fronts. 
\end{itemize}
To emphasize its meaning, it is stated for a critical point $e$ which not necessarily $0_{\rr^d}$ (therefore hypothesis \cref{hyp_crit_point} is not called upon but stated with words), and under a coercivity at infinity condition which is not explicitly stated (hypothesis \cref{hyp_coerc} could be replaced with any other condition ensuring the boundedness of the solutions of the parabolic system \cref{parabolic_system}). By contrast, the proof will be carried out assuming hypotheses \cref{hyp_crit_point,hyp_coerc}, and for the critical point $0_{\rr^d}$. 
\begin{theorem}
\label{thm:main}
For every potential function $V$ in $\ccc^2(\rr^d,\rr)$ which is coercive at infinity and every critical point $e$ point of $V$ which is not a global minimum point, there exists a travelling front, invading $e$, and which is:
\begin{itemize}
\item either \emph{pulled} (therefore travelling at the linear spreading speed of $e$), 
\item or \emph{pushed} and travelling at a speed greater than or equal to this linear spreading speed.
\end{itemize}
\end{theorem}
As mentioned in the remark following \cref{def:pushed_pulled_travelling_front}, the qualifier \emph{pulled} called upon in this theorem must be understood in the broad sense, encompassing the pushmi-pullyu/variational subclass of pulled fronts. 
\subsection{Principle of the proof and bibliographical comments}
Under the stronger assumption that the critical point under consideration is a nondegenerate (local) minimum point of $V$, existence of a front invading this minimum point was proved in \cite[Corollary~1]{Risler_globCVTravFronts_2008} and independently in \cite{AlikakosKatzourakis_heteroclinicTW_2011}; in this case, the maximal linear invasion speed (\cref{def:max_linear_invasion_speed}) is zero, so that every front invading this minimum point is pushed in the sense of \cref{def:pushed_pulled_travelling_front}. \Cref{thm:main} is therefore an extension of these results to the case where the critical point is not necessarily a local minimum point. In both references, the proof relies on a variational structure in travelling frames which is known for long, \cite{FifeMcLeod_approachTravFront_1977}, although attempts to fully embrace its implications are more recent, notably following \cite{Muratov_globVarStructPropagation_2004}; for more detailed bibliographical comments, see for instance \cite{OliverBonafouxRisler_globCVPushedTravFronts_2023} and references therein.

Corollary~1 of \cite{Risler_globCVTravFronts_2008} was obtained as a consequence of a result about global convergence towards travelling fronts invading a non-degenerate local minimum point (Theorem~1 of the same reference). Using a similar approach and crucial ideas introduced in \cite{GallayJoly_globStabDampedWaveBistable_2009}, this result was recently extended to the pushed invasion of a critical point, \cite{OliverBonafouxRisler_globCVPushedTravFronts_2023}. Again, existence of pushed fronts (\cite[Theorem~2]{OliverBonafouxRisler_globCVPushedTravFronts_2023}) follows as a corollary, provided that invasion might occur at a speed which is larger than the maximal linear invasion speed of the invaded critical point. A similar existence result had already been obtained for gradient systems in infinite cylinders, \cite{LuciaMuratovNovaga_existTWInvasionGLCylinders_2008}, although not as a corollary of a global convergence result. 

The scope of the present paper thus reduces to the remaining case where invasion does \emph{not} occur at a speed larger than the maximal linear invasion speed (hypothesis \cref{hyp_pulled} below). Only the existence of a pulled or pushed travelling front is pursued, and for this purpose, the following basic strategy turns out to be sufficient: a $C^1$-small (although not $C^2$-small) perturbation of the potential in the vicinity of the invaded critical point allows to turn this critical point into a local minimum point (for a broader analysis of such perturbations as a basis for pulled fronts selection, see \cite{PaquetteChen_structStabRenGroupPropFronts_1994}). For this perturbed potential, \cite[Theorem~1]{Risler_globCVTravFronts_2008} provides the existence of a pushed travelling front invading this critical point; and a compactness argument provides the intended front (for the unperturbed potential) as a limit when the size and scope of the perturbation go to $0$. The sole difficulties are to prove that the profiles of the fronts for the perturbed potentials do not reduce to smaller and smaller neighbourhoods of the invaded critical point as the perturbation goes to $0$, and that the limit front (the speed of which is, by construction, equal to $\cLinMax$) is either pushed or pulled. This is achieved by an appropriate choice of the perturbation: basically, by choosing the perturbation large enough in the $C^2$-topology so that the Hessian matrix of $0_{\rr^d}$ for the perturbed potential is not only positive definite, but with large enough eigenvalues (see the comments in \cref{subsec:travelling_fronts_perturbed_potentials}). 

It would be way more satisfactory to derive the existence of a pulled front (in the absence of pushed fronts) as a consequence of a global convergence result, as was achieved in \cite{Risler_globCVTravFronts_2008} for invasion of local minimum points and in \cite{OliverBonafouxRisler_globCVPushedTravFronts_2023} for the pushed invasion of critical points. Indeed, one may conjecture that, under the assumptions \cref{hyp_coerc,hyp_crit_point}, for every solution invading $0_{\rr^d}$ with a profile that approaches $0_{\rr^d}$ fast enough to the right end of space, invasion occurs at some well defined speed which is at least equal to $\cLinMax$ and through profiles of either pulled or pushed travelling fronts (if this was known, the existence result \cref{thm:main} would once again follow as a corollary). But the variational approach used in \cite{Risler_globCVTravFronts_2008,OliverBonafouxRisler_globCVPushedTravFronts_2023} does not easily apply to pulled invasion, to begin with because the energy of a pulled front is in most cases equal to $-\infty$ (except for the pushmi-pullyu/variational pulled class). Pushed invasion, which presents strong similarities with the ``local minimum invasion'' case because the speed of the invasion is large enough to prevent the instability of the invaded critical point to develop, is in this respect much easier to tackle. And the global convergence result mentioned above, expected to occur even for pulled invasion (and related, in a broader context, to the ``marginal stability conjecture'', see \cite{AveryScheel_universalSelectionPulledFronts_2022} and references therein)  is currently, to the best knowledge of the authors, an open question. At least, the existence result proved in the present paper provides an additional support, for the parabolic gradient systems considered here, for its plausibility, and ensures the validity of a key assumption made in recent progresses towards its proof (see the remark at the end of the next \namecref{sec:preliminaries}).
\section{Preliminaries: energy in travelling frames and maximal nonlinear invasion speed}
\label{sec:preliminaries}
Let $V$ denote a potential function in $\ccc^2(\rr^d,\rr)$ satisfying assumptions \cref{hyp_coerc,hyp_crit_point}. Let $c$ denote a positive quantity, and let us consider the following weighted Sobolev space:
\[
\begin{aligned}
H^1_c(\rr,\rr^d) &= \bigl\{ w \in \HoneLoc(\rr,\rr^d) : \text{ the functions } \xi\mapsto e^{\frac{1}{2}c\xi} w(\xi) \\
\nonumber
&\qquad\qquad\text{and } \xi\mapsto e^{\frac{1}{2}c\xi} w'(\xi) \text{ are in } L^2(\rr,\rr^d)\bigr\}
\,.
\end{aligned}
\]
Everywhere in the paper, for $u$ in $\rr^d$, the usual Euclidean norm of $u$ is denoted by $\abs{u}$, and $\abs{u}^2$ is simply written as $u^2$. 
\begin{definition}[energy in a frame travelling at the speed $c$]
\label{def:eee_c}
For every $w$ in $H^1_c(\rr,\rr^d)$, let us call \emph{energy (Lagrangian) of $w$ in the frame travelling as speed $c$}, and let us denote by $\eee_{c,V}[w]$, the quantity defined by the integral:
\begin{equation}
\label{def_eee_c}
\eee_{c,V}[w] = \int_{\rr}e^{c\xi}\left(\frac{1}{2} w'(\xi)^2 + V\bigl(w(\xi)\bigr)\right)\, d\xi
\,.
\end{equation}
\end{definition}
Let us consider the quantity
\[
\mathfrak{I}(c) = \inf_{w\in H^1_c(\rr,\rr^d)} \eee_c[w]
\,,
\]
and the subsets $\ccc_{-\infty}$ and $\ccc_0$ of $(0,+\infty)$, defined as
\begin{equation}
\label{def_C_minus_infty_C_0}
\ccc_{-\infty} = \{c\in(0,+\infty): \mathfrak{I}(c) = -\infty \}
\quad\text{and}\quad
\ccc_0 = \{c\in(0,+\infty): \mathfrak{I}(c) = 0 \}
\,.
\end{equation}
It turns out (\cite[Corollary~1.12]{OliverBonafouxRisler_globCVPushedTravFronts_2023}) that there exists a (finite) positive quantity $\cNonLinMax$ such that
\begin{equation}
\label{full_properties_ccc_minus_infty_and_ccc_0}
\ccc_{-\infty} = (0,\cNonLinMax)
\quad\text{and}\quad
\ccc_0 = [\cNonLinMax,+\infty)
\,;
\end{equation}
in addition (\cite[Proposition~2.6]{OliverBonafouxRisler_globCVPushedTravFronts_2023}), 
\begin{equation}
\label{cLinMax_less_or_equal_than_cNonLinMax}
(0,\cLinMax)\subset\ccc_{-\infty}
\,,\quad\text{so that}\quad
\cLinMax\le\cNonLinMax
\,.
\end{equation}
\begin{definition}[maximal nonlinear invasion speed]
\label{def:max_nonlin_invasion_speed}
Let us call \emph{maximal nonlinear invasion speed of the critical point $0_{\rr^d}$} the quantity $\cNonLinMax$ defined by the equalities \cref{full_properties_ccc_minus_infty_and_ccc_0}.
\end{definition}
According to \cite[Theorem~2]{OliverBonafouxRisler_globCVPushedTravFronts_2023}, if 
\[
\cNonLinMax > \cLinMax
\,,
\]
then there exists a pushed front invading $0_{\rr^d}$ at the speed $\cNonLinMax$. In this case, the conclusions of \cref{thm:main} are therefore automatically satisfied. Thus, in view of the second assertion of \cref{cLinMax_less_or_equal_than_cNonLinMax}, it is sufficient to prove \cref{thm:main} under the following additional assumption:
\begin{gather} 
\tag{$\text{H}_{\text{nl-max}=\text{l-max}}$}
\cNonLinMax = \cLinMax
\,.
\label{hyp_pulled}
\end{gather}
It follows from this last assumption that 
\begin{equation}
\label{mu_1_negative}
\cLinMax > 0 
\,,\quad\text{or equivalently}\quad
\mu_1 < 0
\,,
\end{equation}
and that
\begin{equation}
\label{expression_cLinMax}
\cLinMax = 2\sqrt{-\mu_1}
\,,\quad\text{or equivalently}\quad
\mu_1 = - \frac{\cLinMax^2}{4}
\,.
\end{equation}
\begin{remarks}
If $\cNonLinMax$ equals $\cLinMax$, \cref{thm:main} ensures the existence of a front, either pulled \emph{or} pushed, travelling at the speed $\cLinMax$. However, a legitimate question is to wonder whether, in this case, the stronger conclusion stating the existence of a \emph{pulled} front always holds (as everywhere else, ``pulled'' is to be understood in the broader sense including the pushmi-pullyu/variational class, see the remark following \cref{def:pushed_pulled_travelling_front}). Three arguments support (when $\cNonLinMax$ equals $\cLinMax$) the plausibility of this stronger conclusion: 
\begin{enumerate}
\item Generically (for a generic potential $V$), there is no pushed front at the speed $\cLinMax$ \cite{JolyOliverBRisler_genericTransvPulledPushedTFParabGradSyst_2023}, so that, if $\cNonLinMax$ equals $\cLinMax$, this stronger conclusion follows from \cref{thm:main}.
\item If a pushed front travelling at the speed $\cLinMax$ exists, it can be argued that this pushed front is only marginally stable around its leading edge, precisely because the linear spreading speed of the instabilities at the invaded critical point is equal to the speed of this front, see \cite{OliverBonafouxRisler_multiplicityPushedFronts_2023} for a related calculation; 
\item Last, the authors have not been able to produce an example where $\cNonLinMax$ equals $\cLinMax$ and no pulled front exists.
\end{enumerate}
Unfortunately, this stronger conclusion does not follow from the proof provided in this paper and is therefore, to the best knowledge of the authors, an open question. 

Notice also that the existence of a pulled front is a key hypothesis in recent attempts to prove (under additional spectral assumptions) its global stability (in other words, to progress towards the marginal stability conjecture), \cite{AveryScheel_asymptStabCritPulledFrontsResExpansions_2021,AveryScheel_universalSelectionPulledFronts_2022}. Those references deal with (a broad class of) scalar equations, but (as claimed by their authors) there is little doubt that their results can be extended to systems. To be more precise, combining \cref{thm:main} with the conclusions of \cite{JolyOliverBRisler_genericTransvPulledPushedTFParabGradSyst_2023}, it follows that, for a \emph{generic} potential satisfying assumptions \cref{hyp_coerc}, \cref{hyp_crit_point}, and \cref{hyp_pulled}, there exists (for the gradient system \cref{parabolic_system}) a pulled travelling front (in the restricted sense, that is excluding the pushmi-pullyu/variational class) invading $0_{\rr^d}$ (at the speed $\cLinMax$) for which most of the hypotheses made in \cite{AveryScheel_universalSelectionPulledFronts_2022} (and used in this reference to prove that this front is ``selected'') are fulfilled; namely, hypotheses 1, 2, 3, and the ``non-resonance/transversality'' part of hypothesis 4. The remaining part of hypothesis 4 (no unstable point spectrum) cannot be ensured. However, if it is not fulfilled, that is if the pulled front presents some point spectrum instability, then it is expected that the unstable manifold of this pulled front leads to another pulled front for which (again for a generic potential) point spectrum stability holds (to the best knowledge of the authors, this ``expected'' conclusion is an open question). 
\end{remarks}
\section{Proof}
Let $V$ denote a potential function in $\ccc^2(\rr^d,\rr)$ satisfying assumptions \cref{hyp_coerc}, \cref{hyp_crit_point}, and \cref{hyp_pulled}. Our goal is to prove the existence of a pulled front (\cref{def:pushed_pulled_travelling_front}) invading $0_{\rr^d}$ for this potential. 
\subsection{Local perturbation of the potential}
Let us denote by 
\begin{equation}
\label{canonical_basis}
(\mathfrak{u}_1,\dots,\mathfrak{u}_d)
\end{equation}
the canonical basis of $\rr^d$. Without loss of generality, it may be assumed that
\begin{equation}
\label{hessian_potential_V}
D^2V(0_{\rr^d}) = \diag(\mu_1,\dots,\mu_d)
\,,
\end{equation}
so that, for every $u$ in $\rr^d$, 
\begin{equation}
\label{hessian_potential_with_argument}
D^2V(0_{\rr^d})\cdot u \cdot u = \sum_{j=1}^d \mu_j u_j^2 
\,,
\end{equation}
and recall that, according to hypothesis \cref{hyp_pulled}, the least eigenvalue $\mu_1$ of $D^2V(0_{\rr^d})$ is negative (inequalities \cref{mu_1_negative}). Let $\nu$ denote a positive quantity to be chosen later (see the condition \vref{second_condition_on_nu}), satisfying the inequality
\begin{equation}
\label{first_condition_on_nu}
\mu_1 + \nu >0
\,,\quad\text{or equivalently,}\quad
\nu > -\mu_1 = \abs{\mu_1}
\,,
\end{equation}
and let us consider the (positive definite) quadratic form $q$ on $\rr^d$, defined as: 
\[
q(u) = \frac{1}{2} \nu u^2
\,.
\]
For $j$ in $\{1,\dots,d\}$, let us consider the two quantities $\lambdaPert_{j,\pm}$ defined as
\begin{equation}
\label{def_lambda_pert}
\lambdaPert_{j,\pm} = \frac{\cLinMax}{2} \pm \sqrt{\frac{\cLinMax^2}{4}+\mu_j + \nu}
\,;
\end{equation}
in this notation, the exponent ``pert'' refers to the fact that these quantities represent eigenvalues associated with the perturbed potential $W_{\varepsilon,\delta}$ defined below. 
According to the condition \cref{first_condition_on_nu} on $\nu$ and inequalities \cref{sort_mu_j},
\begin{equation}
\label{signs_of_lambda_left_pm}
\lambdaPert_{d,-} \le \cdots \le \lambdaPert_{1,-} < 0 < \lambdaPert_{1,+} \le \cdots \le \lambdaPert_{d,+}
\,.
\end{equation}
For $u=(u_1,\dots,u_d)$ in $\rr^d$, besides the canonical Euclidean norm $\abs{u}$, let us consider a second norm $\norm{u}$ defined as
\begin{equation}
\label{def_norm_alt}
\norm{u} = (\lambdaPert_{1,+})^{-\frac{1}{2}}\left(\sum_{j=1}^d \lambdaPert_{j,+} u_j^2\right)^{\frac{1}{2}}
\,.
\end{equation}
The reason for introducing this norm will appear in \cref{subsubsec:crossing_potential_barrier}, see \cref{fig:trajectories}; thanks to the factor $(\lambdaPert_{1,+})^{-\frac{1}{2}}$, this second norm $\norm{\cdot}$ is equal to the canonical Euclidean norm on $\spanset(\mathfrak{u_1})$, which will ease the writing of the proof of \cref{prop:nonlinear_invasion_speed_perturbed_potential} below. Let $\chi:\rr\to\rr$ denote a smooth cutoff function satisfying
\begin{equation}
\label{cut_off}
\chi(x) = \left\{
\begin{aligned}
1 \quad\text{if}\quad x\le 0\,,\\
0 \quad\text{if}\quad 1\le x\,,
\end{aligned}
\right.
\quad\text{and, for all $x$ in $\rr$,}\quad
0\le \chi(x)\le 1 
\quad\text{and}\quad
\chi'(x)\le 0
\,.
\end{equation}
Let $\varepsilon$ and $\delta$ denote two small positive quantities, and let us consider the potential function $W_{\varepsilon,\delta}:\rr^d \to\rr$, defined as
\begin{equation}
\label{def_W_epsilon_delta}
W_{\varepsilon,\delta}(u) = \chi\left(\frac{\norm{u}-\varepsilon}{\delta}\right) q(u) + V(u)
\,,
\end{equation}
see \cref{fig:W_epsilon_delta}. 
\begin{figure}[htbp]
\centering
\includegraphics[width=\textwidth]{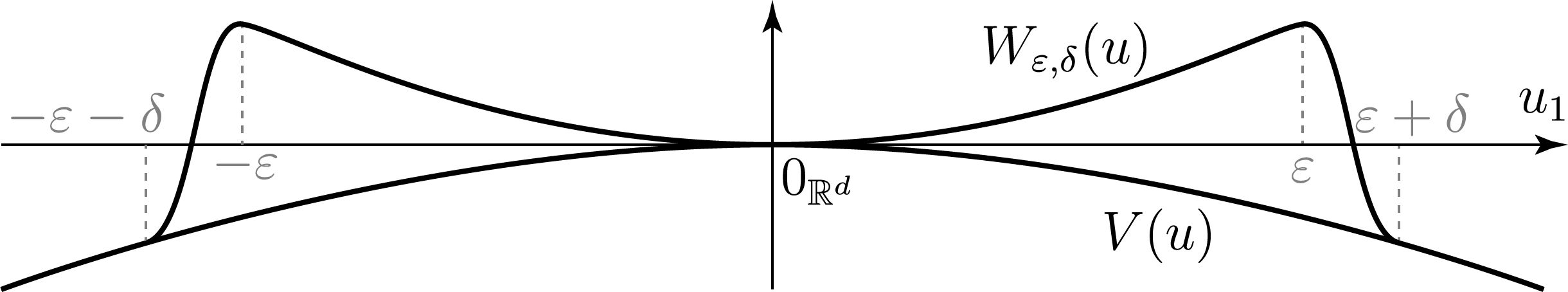}
\caption{Potentials $V$ and $W_{\varepsilon,\delta}$ (the horizontal axis represents the coordinate $u_1$; along this axis the norm $\norm{\cdot}$ is equal to the standard Euclidean norm).}
\label{fig:W_epsilon_delta}
\end{figure}
It follows from this definition that
\begin{equation}
\label{basic_properties_W_minus_V}
W_{\varepsilon,\delta}(0_{\rr^d}) = V(0_{\rr^d})
\quad\text{and, for every $u$ in $\rr^d$, }\quad
W_{\varepsilon,\delta}(u) \ge V(u)
\,,
\end{equation}
and that the Hessian matrix of $W_{\varepsilon,\delta}$ at $0_{\rr^d}$ does not depend on $(\varepsilon,\delta)$; more precisely, according to \cref{hessian_potential_V},
\begin{equation}
\label{hessian_perturbed_potential}
D^2W_{\varepsilon,\delta}(0_{\rr^d}) = \diag(\mu_1+\nu,\dots,\mu_d+\nu)
\,;
\end{equation}
and, according to inequality \cref{first_condition_on_nu}, the least eigenvalue $\mu_1 + \nu$ of this Hessian matrix is positive, so that $0_{\rr^d}$ is a non-degenerate local minimum point for $W_{\varepsilon,\delta}$. 
\subsection{Nonlinear invasion speed for the perturbed potential}
Following the notation of \cref{def:max_linear_invasion_speed,def:max_nonlin_invasion_speed}, let us denote by 
\[
\cLinMax[W_{\varepsilon,\delta}]
\quad\text{and}\quad
\cNonLinMax[W_{\varepsilon,\delta}]
\]
the maximal linear invasion speed and the maximal nonlinear invasion speed (respectively) of the critical point $0_{\rr^d}$ for the potential $W_{\varepsilon,\delta}$. Since $0_{\rr^d}$ is a (nondegenerate) local minimum point of $W_{\varepsilon,\delta}$, it follows (from \cref{def:max_linear_invasion_speed}) that 
\[
\cLinMax[W_{\varepsilon,\delta}] = 0
\,.
\]
In addition, if $\varepsilon$ and $\delta$ are small enough so that $0_{\rr^d}$ is not a global minimum point of $W_{\varepsilon,\delta}$, then it follows from \cite[\GlobCVPushedCorNonEmptinessCminusInfty]{OliverBonafouxRisler_globCVPushedTravFronts_2023} that
\[
0 < \cNonLinMax[W_{\varepsilon,\delta}] 
\,. 
\]
Recall that, according to hypothesis \cref{hyp_pulled}, the same two speeds for the potential $V$ (instead of $W_{\varepsilon,\delta}$) are equal; up to now they were denoted by $\cLinMax$ and $\cNonLinMax$, from now they will always be denoted as $\cLinMax$. It follows from the properties \cref{basic_properties_W_minus_V} (and from definition \cref{def_C_minus_infty_C_0}) that
\[
\ccc_{-\infty}[W_{\varepsilon,\delta}] \subset \ccc_{-\infty}[V]
\,,
\]
and as a consequence, it follows from equalities \cref{full_properties_ccc_minus_infty_and_ccc_0} that
\begin{equation}
\label{cNonLinMax_of_W_not_larger_than_cNonLinMax_of_V}
\cNonLinMax[W_{\varepsilon,\delta}] \le \cLinMax
\,.
\end{equation}
Since $\varepsilon$ and $\delta$ are assumed to be small, the difference $W_{\varepsilon,\delta} - V$ is $\ccc^1$-small, but not $\ccc^2$-small; the support of this difference is included in $\widebar{B}_{\rr^d}(0_{\rr^d},\varepsilon+\delta)$. 
\begin{proposition}[nonlinear invasion speed for the perturbed potential]
\label{prop:nonlinear_invasion_speed_perturbed_potential}
The following limit holds: 
\begin{equation}
\label{nonlinear_invasion_speed_perturbed_potential}
\cNonLinMax[W_{\varepsilon,\delta}] \to \cLinMax
\quad\text{as}\quad
(\varepsilon,\delta)\to(0,0)
\,.
\end{equation}
\end{proposition}
\begin{proof}
Let us consider two quantities $c$ and $\tilde{c}$ satisfying the inequalities
\begin{equation}
\label{inequalities_c_and_tilde_c}
0 < c < \tilde{c} < \cLinMax
\,,
\end{equation}
and let us consider the function $w:\rr\to\rr^d$ defined as
\[
w(\xi) = \left\{
\begin{aligned}
\mathfrak{u}_1 \quad&\text{if}\quad \xi\le 0 \\
e^{-\frac{\tilde{c}}{2}\xi} \mathfrak{u}_1 \quad&\text{if}\quad \xi\ge 0
\end{aligned}
\right.
\,,
\]
see \cref{fig:w} (recall that $\mathfrak{u}_1$ denotes the first vector of the canonical basis of $\rr^d$, see \cref{canonical_basis}). 
\begin{figure}[htbp]
\centering
\includegraphics[width=\textwidth]{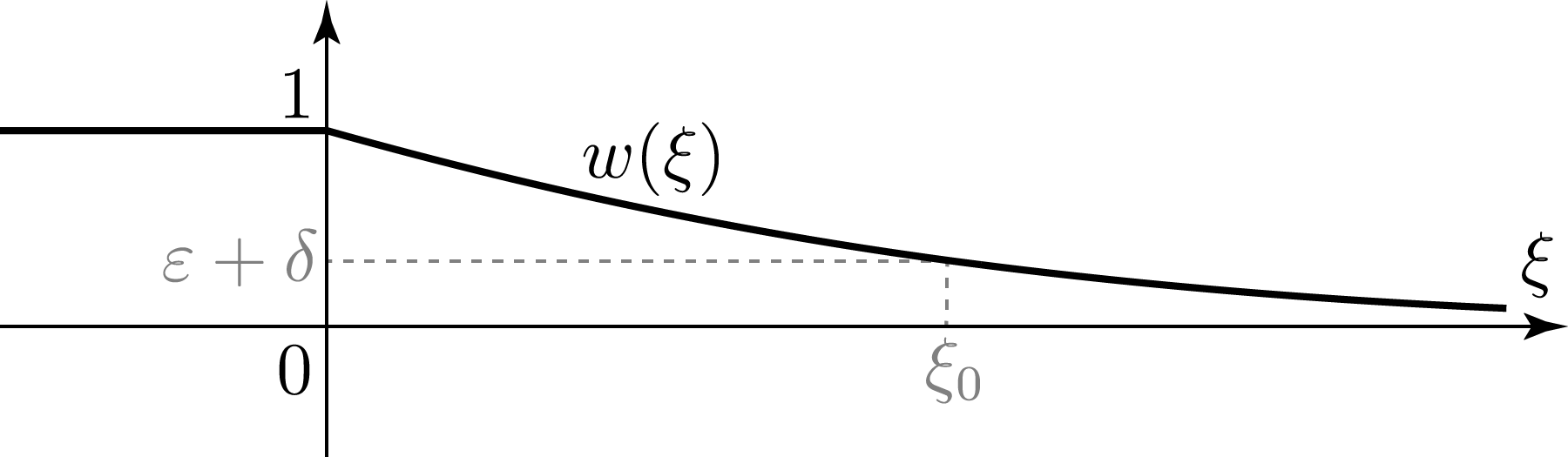}
\caption{Function $\xi\mapsto w(\xi)$.}
\label{fig:w}
\end{figure}
This function $w$ belongs to $H^1_c(\rr,\rr^d)$. In addition, for every positive quantity $\xi$, 
\[
\frac{1}{2} w'(\xi)^2 = \frac{1}{8} \tilde{c}^2 e^{-\tilde{c}\xi}
\,,
\]
and
\[
V\bigl(w(\xi)\bigr) \sim \frac{1}{2} \mu_1 e^{-\tilde{c}\xi} = - \frac{1}{8} \cLinMax^2 e^{-\tilde{c}\xi}
\quad\text{as}\quad
\xi\to+\infty
\,,
\]
so that
\[
e^{c\xi}\left(\frac{1}{2} w'(\xi)^2 + V\bigl(w(\xi)\bigr)\right) \sim - \frac{1}{8} \bigl(\cLinMax^2-\tilde{c}^2\bigr) e^{-(\tilde{c}-c)\xi}
\quad\text{as}\quad
\xi\to+\infty
\,.
\]
It follows that 
\begin{equation}
\label{lim_eee_of_c_V_as_tilde_c_goes_to_c}
\eee_{c,V}[w]\to -\infty 
\quad\text{as}\quad
\tilde{c}\to c 
\,, \quad 
\tilde{c}>c
\,.
\end{equation}
Now let $\varepsilon$ and $\delta$ denote two (small) positive quantities, and let us consider the quantity $\xi_0$ defined as
\[
\xi_0 = \frac{2}{\tilde{c}} \ln\frac{1}{\varepsilon+\delta}
\,,
\quad\text{so that}\quad
e^{-\frac{\tilde{c}}{2}\xi_0} = \varepsilon+\delta
\,,
\]
and let us assume that $\varepsilon$ and $\delta$ are small enough so that $\xi_0$ is positive, see \cref{fig:w}. Observe that, since $w(\xi)$ is proportional to the vector $\mathfrak{u}_1$, $\norm{w(\xi)}$ is equal to $\abs{w(\xi)}$. Thus, 
\[
\norm{w(\xi)} \le \varepsilon + \delta \iff \xi\ge\xi_0
\,,
\]
so that
\[
\begin{aligned}
W_{\varepsilon,\delta}\bigl(w(\xi)\bigr) &= V\bigl(w(\xi)\bigr) \quad\text{for}\quad \xi\le\xi_0\,, \\
\text{and}\quad
W_{\varepsilon,\delta}\bigl(w(\xi)\bigr) &\le V\bigl(w(\xi)\bigr) + q\bigl(w(\xi)\bigr) \quad\text{for}\quad \xi_0\le\xi
\,.
\end{aligned}
\]
It follows that 
\[
\begin{aligned}
0\le \eee_{c,W_{\varepsilon,\delta}}[w] - \eee_{c,V}[w] &= \int_{\rr} e^{c\xi} \Bigl(W_{\varepsilon,\delta}\bigl(w(\xi)\bigr) - V\bigl(w(\xi)\bigr)\Bigr) \, d\xi \\
&= \int_{\xi_0}^{+\infty} e^{c\xi} \Bigl(W_{\varepsilon,\delta}\bigl(w(\xi)\bigr) - V\bigl(w(\xi)\bigr)\Bigr) \, d\xi \\
&\le \int_{\xi_0}^{+\infty} e^{c\xi} q\bigl(w(\xi)\bigr)\, d\xi \\
&= \frac{\nu}{2} \int_{\xi_0}^{+\infty} e^{-(\tilde{c}-c)\xi}\, d\xi \\
&= \frac{\nu}{2} \frac{1}{\tilde{c}-c}(\varepsilon+\delta)^{2\frac{\tilde{c}-c}{\tilde{c}}}
\,.
\end{aligned}
\] 
As a consequence, for fixed $c$ and $\tilde{c}$ satisfying inequalities \cref{inequalities_c_and_tilde_c}, 
\begin{equation}
\label{lim_eee_of_c_W_minus_eee_of_c_V}
\eee_{c,W_{\varepsilon,\delta}}[w] - \eee_{c,V}[w] \to 0
\quad\text{as}\quad
(\varepsilon,\delta)\to(0,0)
\,.
\end{equation}
It follows from the limits \cref{lim_eee_of_c_V_as_tilde_c_goes_to_c,lim_eee_of_c_W_minus_eee_of_c_V} that, for every $c$ in the interval $(0,\cLinMax)$, there exists $\tilde{c}$ in the interval $(c,\cLinMax)$ such that, if the positive quantities $\varepsilon$ and $\delta$ are small enough, then 
\begin{equation}
\label{eee_c_W_of_w_negative}
\eee_{c,W_{\varepsilon,\delta}}[w] < 0
\,;
\end{equation}
and since $w$ belongs to $H^1_c(\rr,\rr^d)$, if this inequality \cref{eee_c_W_of_w_negative} holds, then, according to \cref{def_C_minus_infty_C_0,full_properties_ccc_minus_infty_and_ccc_0},
\[
\cNonLinMax[W_{\varepsilon,\delta}] > c 
\,.
\] 
In view of inequality \cref{cNonLinMax_of_W_not_larger_than_cNonLinMax_of_V}, this completes the proof. 
\end{proof}
\subsection{Pushed travelling fronts for the perturbed potentials}
\label{subsec:travelling_fronts_perturbed_potentials}
It follows from \cite[Theorem~2]{OliverBonafouxRisler_globCVPushedTravFronts_2023} that, if the positive quantities $\varepsilon$ and $\delta$ are small enough so that $\cNonLinMax[W_{\varepsilon,\delta}]$ is positive, then there exists a pushed front travelling at the speed $\cNonLinMax[W_{\varepsilon,\delta}]$ and invading the critical point $0_{\rr^d}$, for the potential $W_{\varepsilon,\delta}$. Let $\xi\mapsto\phi_{\varepsilon,\delta}(\xi)$ denote the profile of this travelling front. It is a bounded global solution of the differential system
\[
\phi'' = -\cNonLinMax[W_{\varepsilon,\delta}] \phi' + \nabla W_{\varepsilon,\delta}(\phi)
\]
satisfying the limit
\[
\phi_{\varepsilon,\delta}(\xi) \to 0_{\rr^d}
\quad\text{as}\quad
\xi\to+\infty
\,.
\]
The profile of the intended pushed or pulled front invading $0_{\rr^d}$ at the speed $\cLinMax$ for the potential $V$ will be obtained as a limit of the profiles $\phi_{\varepsilon,\delta}$, up to appropriate translations of their argument $\xi$ and for the topology of uniform convergence on compact subsets of $\rr$, for a sequence of parameters $(\varepsilon,\delta)$ going to $(0,0)$ (see \cref{subsubsec:obtaining_expected_front}). According to the limit \cref{nonlinear_invasion_speed_perturbed_potential}, such a limit profile must be a solution of the differential system \cref{syst_trav_front_order_2} governing the profiles of waves travelling at the speed $\cLinMax$ for the potential $V$. In order this limit profile to fulfil the conclusions of \cref{thm:main}, two additional features are required. 
\begin{enumerate}
\item There is no guarantee that the limit profile is not constant (for instance, the profiles $\phi_{\varepsilon,\delta}(\xi)$ might converge to $0_{\rr^d}$ as $(\varepsilon,\delta)$ goes to $(0,0)$, uniformly with respect to $\xi$ in $\rr$). Indeed, the critical point $0_{\rr^d}$ is not assumed to be non-degenerate, and is therefore not necessarily isolated as a critical point of $V$. 
\item Even if the limit profile is assumed to be nonconstant, and is the profile of a \emph{front} invading $0_{\rr^d}$ at the speed $\cLinMax$ in the sense of \cref{def:travelling_front_invading_0_speed_c}, there is no guarantee that this travelling front is either pushed or pulled (its profile might approach $0_{\rr^d}$, to the right end of space, at an exponential rate of convergence which is weaker than $\cLinMax/2$, a feature characterizing a ``mild'' travelling front, \cite{JolyOliverBRisler_genericTransvPulledPushedTFParabGradSyst_2023}).
\end{enumerate}
As will turn out from forthcoming arguments, obtaining a limit profile with these two additional features can be ensured by an appropriate choice of the parameters, namely:
\begin{itemize}
\item a ratio $\delta/\varepsilon$ going to $0$ as $(\varepsilon,\delta)\to(0,0)$, 
\item and a large enough (positive) parameter $\nu$.
\end{itemize}
Basically, it follows from such a choice that the profiles $\phi_{\varepsilon,\delta}(\xi)$ exit (as $\xi$ decreases from $+\infty$) the support of the perturbation $W_{\varepsilon,\delta}-V$ with a sufficient ``impulse'' to ensure that the two additional features above hold. Carrying out this analysis is the main purpose of the remaining steps of the proof. 
\subsection{Limit of an appropriate sequence of profiles}
\subsubsection{Sequence}
Let us consider two sequences $(\varepsilon_n)_{n\in\nn}$ and $(\delta_n)_{n\in\nn}$ of positive quantities, small enough so that $\cNonLinMax[W_{\varepsilon_n,\delta_n}]$ is positive for all $n$ in $\nn$, and satisfying
\begin{equation}
\label{properties_sequences_eps_n_delta_n}
\varepsilon_n\to 0 
\quad\text{and}\quad 
\delta_n\to 0
\quad\text{and}\quad 
\frac{\delta_n}{\varepsilon_n} \to 0
\quad\text{as}\quad 
n\to+\infty
\,.
\end{equation}
Let us consider the sequences $(W_n)_{n\in\nn}$ and $(c_n)_{n\in\nn}$ and $(\phi_n)_{n\in\nn}$ of potentials, invasion speeds, and profiles of pushed fronts, defined as
\[
W_n = W_{\varepsilon_n,\delta_n}
\quad\text{and}\quad
c_n = \cNonLinMax[W_n]
\quad\text{and}\quad
\phi_n = \phi_{\varepsilon_n,\delta_n}
\,.
\]
It follows from \cref{prop:nonlinear_invasion_speed_perturbed_potential} that 
\begin{equation}
\label{c_n_goes_to_cLinMax}
c_n \to \cLinMax
\quad\text{as}\quad
n\to+\infty
\,.
\end{equation}
\subsubsection{Space variable reversal}
Let us consider the function $\zeta\mapsto \psi_n(\zeta)$, defined as
\begin{equation}
\label{def_psi_n}
\psi_n(\zeta) = \phi_n(\xi) 
\quad\text{for}\quad
\zeta = -\xi
\,, 
\end{equation}
so that $\psi_n$ is a global solution of the differential system
\begin{equation}
\label{diff_syst_psi_n}
\psi'' = c_n \psi' + \nabla W_n(\psi)
\,,
\end{equation}
for which the positive quantity $c_n$ represents a \emph{negative} damping coefficient, and $\psi_n(\zeta)$ goes to $0_{\rr^d}$ as $\zeta$ goes to $-\infty$. The quantities $\lambdaPert_{j,\pm}$ introduced in \cref{def_lambda_pert} are the eigenvalues of the linearization at $0_{\rr^d}$ of this differential system; according to inequalities \cref{signs_of_lambda_left_pm} (those eigenvalues are real and nonzero), the image of $(\psi_n,\psi_n')$ must belong to the unstable manifold of $0_{\rr^{2d}}$ for the first order differential system derived from \cref{diff_syst_psi_n}, and there must exist a (unique) real quantity $\zetaEps_n$ such that
\begin{equation}
\label{properties_psi_n_left_of_zero}
\norm{\psi_n(\zetaEps_n)} = \varepsilon_n \,,
\quad\text{and, for every $\zeta$ in $(-\infty,\zetaEps_n)$,}\quad
\norm{\psi_n(\zeta)} < \varepsilon_n
\,,
\end{equation}
see \cref{fig:xi_zeta_line}.
\subsubsection{Renormalization and limit of renormalized profiles and potentials}
Let us consider, for $n$ in $\nn$, the ``renormalized'' profiles $\tilde{\psi}_n$ and potentials $\tilde{W}_n$ defined as
\begin{align}
\label{def_tilde_psi_n}
\tilde{\psi}_n(\zeta) &= \frac{1}{\varepsilon_n} \psi_n(\zetaEps_n + \zeta)\,, \\
\nonumber
\text{and}\quad
\tilde{W}_n(u) &= \frac{1}{\varepsilon_n^2} W_n(\varepsilon_n u) = \chi\left(\frac{\norm{u}-1}{\delta_n/\varepsilon_n}\right)q(u) + \frac{1}{\varepsilon_n^2} V(\varepsilon_n u)
\,.
\end{align}
According to the definition \cref{def_tilde_psi_n}, it follows from the properties \cref{properties_psi_n_left_of_zero} that
\begin{equation}
\label{properties_tilde_psi_n_left_of_zero}
\norm{\tilde{\psi}_n(0)} = 1 \,,
\quad\text{and, for every $\zeta$ in $(-\infty,0)$,}\quad
\norm{\tilde{\psi}_n(\zeta)} < 1
\,;
\end{equation}
in addition, dividing the differential system \cref{diff_syst_psi_n} by $\varepsilon_n$, it follows that the function $\tilde{\psi}_n$ is a (global, bounded) solution of the differential system
\[
\tilde{\psi}_n'' = c_n \tilde{\psi}_n' + \nabla \tilde{W}_n(\tilde{\psi}_n)
\,.
\]
Let us consider the potential function $\tilde{W}_{\infty}:\rr^d\to\rr$, defined as
\[
\tilde{W}_{\infty}(u) = \left\{
\begin{aligned}
\frac{1}{2} D^2V(0_{\rr^d})\cdot u \cdot u + q(u) = \frac{1}{2} \left(\nu u^2 + \sum_{j=1}^d \mu_j u_j^2\right) \quad&\text{if}\quad \norm{u}\le 1 \,, \\
\frac{1}{2} D^2V(0_{\rr^d})\cdot u \cdot u = \frac{1}{2} \sum_{j=1}^d \mu_j u_j^2 \quad&\text{if}\quad \norm{u}> 1 
\,,
\end{aligned}
\right.
\]
which presents a discontinuity of amplitude $\frac{1}{2}\nu u^2$ located on the set $\{u\in\rr^d:\norm{u} = 1\}$, see \cref{fig:graph_tilde_W_infty}. 
\begin{figure}[htbp]
\centering
\begin{subfigure}{.52\textwidth}
\includegraphics[width=\textwidth]{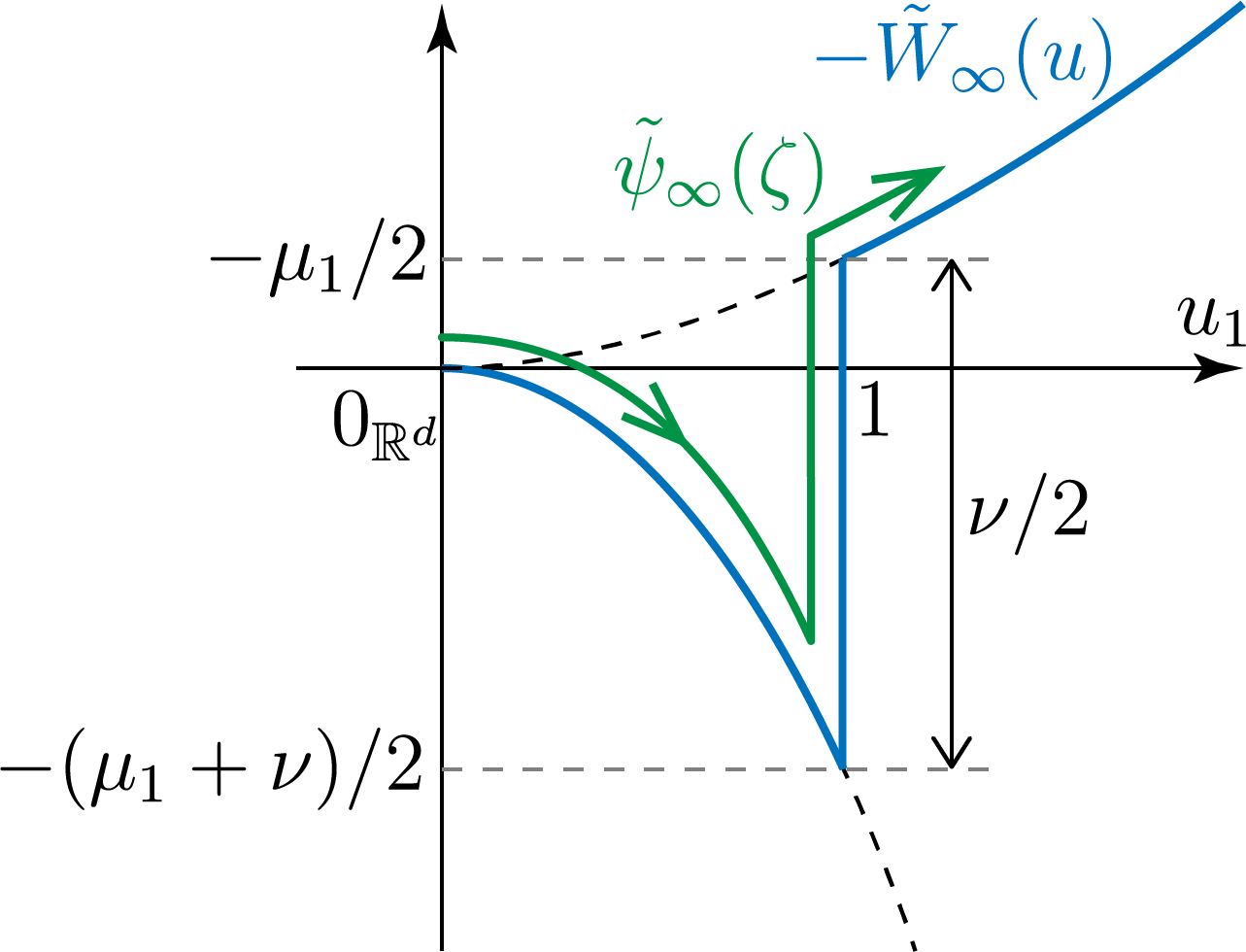}
\caption{Graph of $-\tilde{W}_{\infty}$ (minus the limit of the renormalized potentials) and behaviour of $\zeta\mapsto\tilde{\psi}_\infty(\zeta)$. As on \cref{fig:W_epsilon_delta}, the horizontal axis represents the coordinate $u_1$; along this axis the norm $\norm{\cdot}$ is equal to the standard Euclidean norm $\abs{\cdot}$. The larger the quantity $\nu$, the larger the contribution of the negative damping $\cLinMax$ to the ``impulse'' of $\tilde{\psi}_\infty$ when it reaches the potential barrier at $\norm{u}$ equals $1$.}
\label{fig:graph_tilde_W_infty}
\end{subfigure}
\hspace{.02\textwidth}
\begin{subfigure}{.44\textwidth}
\includegraphics[width=\textwidth]{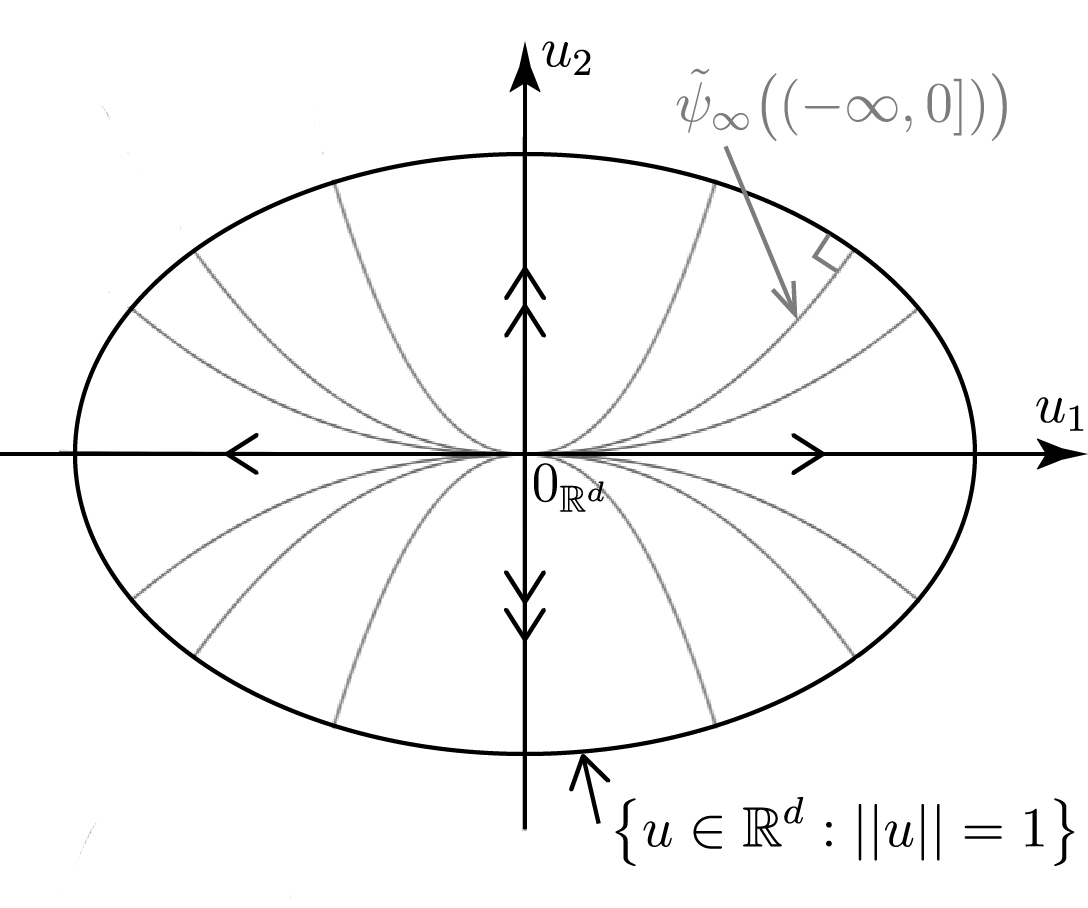}
\caption{Ellipsoid $\bigl\{u\in\rr^d:\norm{u}=1\bigr\}$ in $\rr^d$ and trajectories inside this ellipsoid of solutions in the unstable manifold of $0_{\rr^{2d}}$ for the differential system \cref{limit_renormalized_differential_system}. At the points where these trajectories reach the ellipsoid, the velocity and the tangent plane to the ellipsoid are orthogonal.}
\label{fig:trajectories}
\end{subfigure}
\caption{}
\end{figure}
According to the last property of \cref{properties_sequences_eps_n_delta_n}, the potentials $\tilde{W}_n$ converge towards $\tilde{W}_{\infty}$ as $n$ goes to $+\infty$, uniformly on $\{u\in\rr^d:\norm{u}\le 1\}$, and uniformly on compact subsets of $\{u\in\rr^d:\norm{u}>1\}$. By compactness, up to replacing the sequence $(\varepsilon_n,\delta_n)_{n\in\nn}$ by a subsequence, it may be assumed that there exists a function $\rr\to\rr^d$, $\zeta\mapsto\tilde{\psi}_\infty(\zeta)$, such that:
\begin{equation}
\label{compactness_tilde_psi_n}
\tilde{\psi}_n \to \tilde{\psi}_\infty 
\quad\text{as}\quad
n\to+\infty
\,,
\end{equation}
uniformly on every compact subset of $\rr$. In addition, it follows from the properties \cref{properties_tilde_psi_n_left_of_zero} that
\begin{equation}
\label{basic_propertis_tilde_psi_infty}
\norm{\tilde{\psi}_\infty(0)} = 1 \,,
\quad\text{and, for every $\zeta$ in $(-\infty,0)$,}\quad
\norm{\tilde{\psi}_\infty(\zeta)} \le 1
\,.
\end{equation}
The function $\tilde{\psi}_\infty$ is continuous on $\rr$, and is a solution of the differential system
\begin{equation}
\label{limit_renormalized_differential_system}
\psi'' = \cLinMax \psi' + \nabla \tilde{W}_{\infty}(\psi)
\,,\quad\text{on}\quad
\left\{\zeta\in\rr:\norm{\tilde{\psi}_\infty(\zeta)}\not=1\right\}
\,,
\end{equation}
see \cref{fig:trajectories}. Solutions of this differential system are continuous, but their derivative (speed) is discontinuous at the times where they meet the ellipsoid $\bigl\{u\in\rr^d:\norm{u}=1\bigr\}$ where the potential $\tilde{W}_{\infty}$ is discontinuous. At such times, the component of the velocity which is tangent to the tangent plane to the ellipsoid is unchanged, while the perpendicular component is either reversed (rebound) or increased/decreased to balance for the jump in potential energy induced by the crossing of this potential barrier. In addition, solutions which are in the unstable manifold of $0_{\rr^{2d}}$ for this system (including $\tilde{\psi}_\infty$) reach this ellipsoid with a perpendicular velocity (indeed, by construction, the gradient of the norm $\norm{\cdot}$ is parallel to the derivative of such solutions in the interior of this ellipsoid); see \cref{fig:trajectories}. 
\subsubsection{Asymptotics of the limit of renormalized profiles}
\label{subsubsec:crossing_potential_barrier}
For $j$ in $\{1,\dots,d\}$, let us consider the quantities $\lambda_{j,\pm}$ defined as
\[
\lambda_{j,\pm} = \frac{\cLinMax}{2} \pm \sqrt{\frac{\cLinMax^2}{4}+\mu_j}
\,,
\]
which are the analogues of the quantities $\lambdaPert_{j,\pm}$ defined in \cref{def_lambda_pert}, but for $D^2V(0_{\rr^d})$ which equals $\diag(\mu_1,\dots,\mu_d)$ (equality \cref{hessian_potential_V}), rather than for $D^2V(0_{\rr^d}) +\nu I_d$ which equals $\diag(\mu_1+\nu,\dots,\mu_d+\nu)$. The aim of this \namecref{subsubsec:crossing_potential_barrier} is to prove the following lemma. 
\begin{lemma}[asymptotics of the limit of renormalized profiles]
\label{lem:asymptotics_limit_renormalized_profiles}
If the quantity $\nu$ is large enough (the precise condition is provided by the condition \cref{second_condition_on_nu} below), then the following conclusions hold:
\begin{align}
\label{norm_of_tilde_psi_infty_remains_larger_than_1}
\text{for every $\zeta$ in $(0,+\infty)$,} \quad &\norm{\tilde{\psi}_\infty(\zeta)} > 1 
\,, \\
\label{tilde_psi_infty_goes_to_infty}
\text{and}\quad &\abs{\tilde{\psi}_\infty(\zeta)} \to+\infty
\quad\text{as}\quad \zeta\to+\infty 
\,,
\end{align}
and, for every $j$ in $\{1,\dots,d\}$ such that $\tilde{\psi}_{\infty,j}(0)$ is non zero,
\begin{equation}
\label{limit_of_tilde_psi_injty_j_prime_over_not_prime}
\frac{\tilde{\psi}_{\infty,j}'(\zeta)}{\tilde{\psi}_{\infty,j}(\zeta)} \to \lambda_{j,+}
\quad\text{as}\quad \zeta\to+\infty 
\,.
\end{equation}
\end{lemma}
\begin{proof}
The main step is to show that, provided that the quantity $\nu$ is large enough, $\tilde{\psi}_\infty$ crosses the potential barrier it is facing at $\zeta$ equals $0$ (rather than ``rebounding'' on it) and keeps after this crossing a large enough kinetic energy to ensure a further growth at an exponential rate which is not smaller than $\cLinMax/2$ (inequality \cref{lower_bound_psi_prime_over_psi_prelim} below).

It follows from the second property of \cref{basic_propertis_tilde_psi_infty} that $\tilde{\psi}_\infty$ is (according to the limit \cref{c_n_goes_to_cLinMax}) a solution on $(-\infty,0)$ of the (linear, diagonal) differential system
\begin{equation}
\label{diff_syst_linearized_increased_potential}
\tilde{\psi}_\infty'' = \cLinMax \tilde{\psi}_\infty' + \bigl(D^2 V(0_{\rr^d}) + \nu I_d\bigr)\cdot \tilde{\psi}_\infty
\,,
\end{equation}
which is nothing but system \cref{limit_renormalized_differential_system} in the domain $\bigl\{u\in\rr^d:\norm{u}<1\bigr\}$. Let 
\[
\bigl(\tilde{\psi}_{\infty,1},\dots,\tilde{\psi}_{\infty,d}\bigr)
\]
denote the components of $\tilde{\psi}_\infty$ in $\rr^d$. For every $j$ in $\{1,\dots,d\}$, the real-valued function $\tilde{\psi}_{\infty,j}$ is a solution, on $(-\infty,0)$, of the linear differential equation
\begin{equation}
\label{linear_by_components_renormalization_at_varepsilon}
\psi'' = \cLinMax \psi' + (\mu_j +\nu)\psi 
\,,
\end{equation}
associated with the eigenvalues $\lambdaPert_{j,\pm}$ defined in \cref{def_lambda_pert}. It follows from inequalities \cref{signs_of_lambda_left_pm} and from the properties \cref{properties_tilde_psi_n_left_of_zero} that, for every $\zeta$ in $(-\infty,0]$, 
\[
\tilde{\psi}_{\infty,j}(\zeta) = \tilde{\psi}_{\infty,j}(0) e^{\lambdaPert_{j,+}\zeta}
\,.
\]
In particular, 
\begin{equation}
\label{tilde_psi_infty_prime_of_zero_minus}
\tilde{\psi}_{\infty,j}'(0^-) = \lambdaPert_{j,+} \tilde{\psi}_{\infty,j}(0)
\,.
\end{equation}
According to the first property of \cref{basic_propertis_tilde_psi_infty}, the quantity $\norm{\tilde{\psi}_\infty(0)}$ must be equal to $1$. It follows that there exists at least one integer $j$ in $\{1,\dots,d\}$ such that the function $\tilde{\psi}_{\infty,j}$ is not identically equal to $0$. For such an integer $j$, the quantity $\tilde{\psi}_{\infty,j}(0)$ is therefore also nonzero, and the kinetic energy of the solution $\tilde{\psi}_{\infty,j}(\zeta)$ at $\zeta$ equals $0^-$ is equal to 
\[
\frac{1}{2}(\lambdaPert_{j,+})^2\tilde{\psi}_{\infty,j}(0)^2
\,,
\]
so that the kinetic energy of the solution $\tilde{\psi}_{\infty}(\zeta)$ (of the differential system \cref{diff_syst_linearized_increased_potential}) at $\zeta$ equals $0^-$ is equal to 
\begin{equation}
\label{kinetic_energy}
\frac{1}{2} \sum_{j=1}^d (\lambdaPert_{j,+})^2\tilde{\psi}_{\infty,j}(0)^2
\,,
\end{equation}
while the ``potential barrier'' faced by the same solution $\tilde{\psi}_{\infty}(\zeta)$, still at $\zeta$ equals $0^-$, is equal to
\begin{equation}
\label{potential_barrier}
\frac{1}{2} \nu \tilde{\psi}_{\infty}(0)^2 = \frac{1}{2} \sum_{j=1}^d \nu \tilde{\psi}_{\infty,j}(0)^2
\,.
\end{equation}
Up to the nonnegative factor $\tilde{\psi}_{\infty,j}(0)^2$, the difference between the $j$-th term of the sum in \cref{kinetic_energy} and the $j$-th term of the sum in \cref{potential_barrier} is equal (according to the expression \cref{def_lambda_pert} of $\lambdaPert_{j,+}$) to
\[
(\lambdaPert_{j,+})^2 - \nu = \frac{\cLinMax^2}{4} + \cLinMax\sqrt{\frac{\cLinMax^2}{4}+\mu_j + \nu} + \left(\frac{\cLinMax^2}{4} + \mu_j\right)
\,.
\]
According to the expression \cref{expression_cLinMax} of the quantity $\cLinMax$ and to the condition \cref{first_condition_on_nu} on $\nu$, both quantities $\cLinMax^2/4+\mu_j$ and $\mu_j + \nu$ are nonnegative (and the second one is positive). It follows that
\[
(\lambdaPert_{j,+})^2 - \nu \ge \frac{3}{4}\cLinMax^2 
\,.
\]
As a consequence, the difference between the kinetic energy \cref{kinetic_energy} and the potential barrier \cref{potential_barrier} satisfies the inequality
\begin{equation}
\label{balance_kinetic_energy_potential_barrier}
\begin{aligned}
\frac{1}{2} \sum_{j=1}^d (\lambdaPert_{j,+})^2\tilde{\psi}_{\infty,j}(0)^2 - \frac{1}{2} \sum_{j=1}^d \nu \tilde{\psi}_{\infty,j}(0)^2 
&> \frac{3}{8}\cLinMax^2 \sum_{j=1}^d \tilde{\psi}_{\infty,j}(0)^2  \\
&= \frac{3}{8}\cLinMax^2 \tilde{\psi}_{\infty}(0)^2
\,,
\end{aligned}
\end{equation}
and is in particular positive. Now the key point is that, according to the definition \cref{def_norm_alt} of the norm $\norm{\cdot}$, the velocity $\tilde{\psi}_{\infty}'(0^-)$ is perpendicular to the level set of $\norm{\cdot}$ at $\tilde{\psi}_{\infty}(0)$, and is therefore perpendicular to the potential barrier at $\zeta$ equals $0$. As a consequence, due to the positivity of the difference \cref{balance_kinetic_energy_potential_barrier}, the solution crosses this potential barrier and this crossing results in a discontinuity of the velocity but a continuity of its orientation and direction. In other words, the velocity $\tilde{\psi}_{\infty}'(0^+)$, after crossing this potential barrier is proportional to the velocity $\tilde{\psi}_{\infty}'(0^-)$ before this crossing, with a proportionality coefficient (denoted by $\gamma$) which is in $(0,1)$; with symbols, 
\begin{equation}
\label{tilde_psi_infty_prime_of_zero_plus}
\begin{aligned}
\tilde{\psi}_{\infty}'(0^+) &= \gamma \tilde{\psi}_{\infty}'(0^-) \,, \\
\text{or in other words, for every $j$ in $\{1,\dots,d\}$,}\quad
\tilde{\psi}_{\infty,j}'(0^+) &= \gamma \tilde{\psi}_{\infty,j}'(0^-)
\,.
\end{aligned}
\end{equation}
It follows that the kinetic energy at $\zeta$ equals $0^+$ reads
\[
\frac{1}{2}\tilde{\psi}_{\infty}'(0^+)^2 = \frac{1}{2} \gamma^2 \sum_{j=1}^d (\lambdaPert_{j,+})^2\tilde{\psi}_{\infty,j}(0)^2
\,.
\]
Since this kinetic energy is equal to the left hand term of inequality \cref{balance_kinetic_energy_potential_barrier}, it follows that
\begin{equation}
\label{lower_bound_kinetic_energy}
\frac{1}{2} \gamma^2 \sum_{j=1}^d (\lambdaPert_{j,+})^2\tilde{\psi}_{\infty,j}(0)^2 > \frac{3}{8}\cLinMax^2 \tilde{\psi}_{\infty}(0)^2
\,.
\end{equation}
Now, according to inequalities \cref{signs_of_lambda_left_pm},
\[
\tilde{\psi}_{\infty}(0)^2 = \sum_{j=1}^d \tilde{\psi}_{\infty,j}(0)^2 \ge (\lambdaPert_{d,+})^{-2} \sum_{j=1}^d (\lambdaPert_{j,+})^2\tilde{\psi}_{\infty,j}(0)^2
\,.
\]
and it follows from inequality \cref{lower_bound_kinetic_energy} that
\[
\frac{1}{2} \gamma^2 (\lambdaPert_{d,+})^2 \tilde{\psi}_\infty(0)^2 > \frac{3}{8}\cLinMax^2 \tilde{\psi}_{\infty}(0)^2 \,,
\quad\text{or equivalently,}\quad
\gamma  > \frac{1}{\lambdaPert_{d,+}} \frac{\sqrt{3}}{2} \cLinMax
\,,
\]
and as a consequence, for every $j$ in $\{1,\dots,d\}$, 
\begin{equation}
\label{lower_bound_gamma}
\gamma \lambdaPert_{j,+} > \frac{\lambdaPert_{j,+}}{\lambdaPert_{d,+}} \frac{\sqrt{3}}{2} \cLinMax \ge \frac{\lambdaPert_{1,+}}{\lambdaPert_{d,+}} \frac{\sqrt{3}}{2} \cLinMax
\,.
\end{equation}
Observe that
\[
\frac{\lambdaPert_{1,+}}{\lambdaPert_{d,+}} \to 1 
\quad\text{as}\quad
\nu\to+\infty
\,,
\]
so that, provided that the positive quantity $\nu$ is chosen large enough, 
\begin{equation}
\label{second_condition_on_nu}
\frac{\lambdaPert_{1,+}}{\lambdaPert_{d,+}} > \frac{1}{\sqrt{3}}
\,.
\end{equation}
When this condition is satisfied, it follows from inequalities \cref{lower_bound_gamma} that, for every $j$ in $\{1,\dots,d\}$, 
\[
\gamma \lambdaPert_{j,+} > \frac{1}{2} \cLinMax
\,,
\]
and it therefore follows from equalities \cref{tilde_psi_infty_prime_of_zero_minus,tilde_psi_infty_prime_of_zero_plus} that, for every $j$ in $\{1,\dots,d\}$ such that $\tilde{\psi}_{\infty,j}(0)$ is nonzero, 
\begin{equation}
\label{lower_bound_psi_prime_over_psi_prelim}
\frac{\tilde{\psi}_{\infty,j}'(0^+)}{\tilde{\psi}_{\infty,j}(0^+)} > \frac{1}{2} \cLinMax 
\,.
\end{equation}
Let us consider the quantity $\zetaBack$ defined as
\[
\zetaBack = \inf\left\{\zeta\in(0,+\infty): \norm{\tilde{\psi}_{\infty}(\zeta)} = 1 \right\}
\,,
\]
with the convention that the infimum of an empty subset of $\rr$ is equal to $+\infty$; thus $\zetaBack$ is in $(0,+\infty)\cup\{+\infty\}$. On the interval $(0,\zetaBack)$, $\tilde{\psi}_{\infty}$ is a solution of the linear differential system
\begin{equation}
\label{linear_differential_system_unperturbed}
\psi'' = \cLinMax \psi' + D^2 V(0_{\rr^d}) \psi
\,,
\end{equation}
and, for every $j$ in $\{1,\dots,d\}$, the $j$-th component $\tilde{\psi}_{\infty,j}$ of $\tilde{\psi}_{\infty}$ is a solution of the linear differential equation 
\begin{equation}
\label{linear_differential_equation_unperturbed}
\psi'' = \cLinMax \psi' + \mu_j \psi
\,,
\end{equation}
the eigenvalues of which are:
\begin{equation}
\label{lambda_j_plus_minus}
\lambda_{j,\pm} = \frac{\cLinMax}{2} \pm \sqrt{\frac{\cLinMax^2}{4}+\mu_j}
\,,\quad j\in\{1,\dots,d\}
\,.
\end{equation}
It thus follows from inequality \cref{lower_bound_psi_prime_over_psi_prelim} that, for every $j$ in $\{1,\dots,d\}$ such that $\tilde{\psi}_{\infty,j}(0)$ is nonzero, 
\begin{equation}
\label{lower_bound_psi_prime_over_psi}
\frac{\tilde{\psi}_{\infty,j}'(0^+)}{\tilde{\psi}_{\infty,j}(0^+)} > \max\bigl(\lambda_{j,-},0\bigr)
\,.
\end{equation}
As can be seen on the phase portraits of the linear differential equations \cref{linear_differential_equation_unperturbed} displayed on \cref{fig:phase_portraits}, this inequality \cref{lower_bound_psi_prime_over_psi} ensures that, for every $j$ in $\{1,\dots,d\}$ such that $\tilde{\psi}_{\infty,j}(0)$ is nonzero, the function $\abs{\tilde{\psi}_{\infty,j}}$ is increasing on the whole interval $(0,\zetaBack)$; this shows that the quantity $\zetaBack$ is actually equal to $+\infty$, and proves conclusion \cref{norm_of_tilde_psi_infty_remains_larger_than_1}. In addition, again in view of these phases portraits, the intended inequality \cref{limit_of_tilde_psi_injty_j_prime_over_not_prime} follows, and as a consequence the remaining property \cref{tilde_psi_infty_goes_to_infty} holds. All the conclusions of \cref{lem:asymptotics_limit_renormalized_profiles} are proved. 
\begin{figure}[htbp]
\centering
\includegraphics[width=\textwidth]{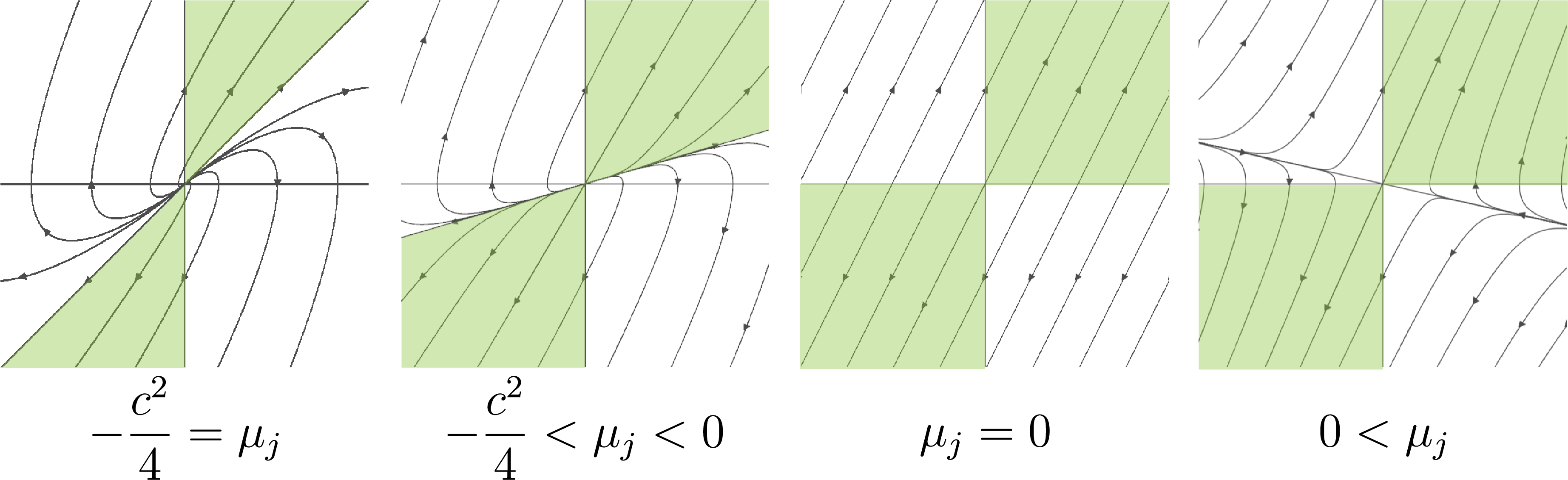}
\caption{Phase portrait of the linear differential equation \cref{linear_differential_equation_unperturbed}, in the various cases defined by the sign and value of the eigenvalue $\mu_j$. The green areas are the ones for which the inequality \cref{lower_bound_psi_prime_over_psi} is satisfied. In each case $c$ is equal to $2$ and $\mu_j$ takes, from left to right, the values: $-1$, $-1/2$, $0$, $1/2$. The slopes of the two eigenlines are $\lambda_{j,-}$ (which decreases from left to right) and $\lambda_{j,+}$ (which increases from left to right). Credits to \url{https://mathlets.org}!}
\label{fig:phase_portraits}
\end{figure}
\end{proof}
\subsubsection{Positively invariant cone around the half-speed unstable subspace}
Let $j_0$ denotes the largest integer in $\{1,\dots,d\}$ such that $\mu_{j_0}=\mu_1$ (generically $j_0$ equals $1$). According to their expression \cref{lambda_j_plus_minus}, the eigenvalues of the linear differential system \cref{linear_differential_system_unperturbed} are ordered as follows: 
\[
\begin{aligned}
&\lambda_{d,-} \le \cdots \le \lambda_{j_0+1,-} < \lambda_{j_0,-} = \cdots = \lambda_{1,-} = \frac{\cLinMax}{2} \,, \\
\text{and}\quad
&\lambda_{d,+} \ge \cdots \ge \lambda_{j_0+1,+} > \lambda_{j_0,+} = \cdots = \lambda_{1,+} = \frac{\cLinMax}{2} 
\,.
\end{aligned}
\]
Let us denote:
\begin{itemize}
\item by $\eeemu$ the ``mild'' unstable subspace of $\rr^{2d}$, which is the sum of all eigenspaces associated with eigenvalues that are less than $\cLinMax/2$,
\item and by $\eeehsu$ the ``half speed'' unstable subspace of $\rr^{2d}$, which is the sum of all eigenspaces (and generalized eigenspaces) associated with eigenvalues that are greater than or equal to $\cLinMax/2$, 
\end{itemize}
for the linear differential system \cref{linear_differential_system_unperturbed}. Then, the subspace $\eeemu$ is spanned by the vectors
\[
\begin{pmatrix} \mathfrak{u}_d \\  \lambda_{d,-}\mathfrak{u}_d \end{pmatrix}
,\dots,
\begin{pmatrix} \mathfrak{u}_{j_0+1} \\  \lambda_{j_0+1,-}\mathfrak{u}_{j_0+1} \end{pmatrix},
\,,
\]
and $\eeehsu$ is spanned by the vectors
\[
\begin{pmatrix} \mathfrak{u}_1 \\ \frac{\cLinMax}{2} \mathfrak{u}_1 \end{pmatrix},
\begin{pmatrix} 0 \\  \mathfrak{u}_1 \end{pmatrix},
\dots,
\begin{pmatrix} \mathfrak{u}_{j_0} \\ \frac{\cLinMax}{2} \mathfrak{u}_{j_0} \end{pmatrix},
\begin{pmatrix} 0 \\  \mathfrak{u}_{j_0} \end{pmatrix},
\begin{pmatrix} \mathfrak{u}_{j_0+1} \\  \lambda_{j_0+1,+}\mathfrak{u}_{j_0+1} \end{pmatrix},
\dots,
\begin{pmatrix} \mathfrak{u}_d \\  \lambda_{d,+}\mathfrak{u}_d \end{pmatrix}
\,;
\]
in particular, 
\[
\dim\eeehsu = d + j_0 
\quad\text{and}\quad
\dim\eeemu = d - j_0 
\quad\text{and}\quad
\eeehsu \oplus \eeemu = \rr^{2d}
\,.
\]
Let $\alpha$ denote a positive quantity, small enough so that
\begin{equation}
\label{gap_between_mu_and_hsu}
\lambda_{j_0+1,-} < \frac{\cLinMax}{2} -\alpha
\end{equation}
and
\begin{equation}
\label{gap_between_0_and_hsu}
0 < \frac{\cLinMax}{2} - 3\alpha
\,,
\end{equation}
and let us consider the $2d\times2d$-matrix
\[
A = \begin{pmatrix}
0 & I_d \\
D^2V(0_{\rr^d}) & \cLinMax
\end{pmatrix}
\]
associated with the linear differential system \cref{linear_differential_system_unperturbed}. According to the definitions of the subspaces $\eeehsu$ and $\eeemu$, there exist a norm $\norm{\cdot}_\hsu$ on $\eeehsu$ and a norm $\norm{\cdot}_\mutext$ on $\eeemu$ such that, for every positive quantity $\zeta$, the following properties hold:
\begin{equation}
\label{inequalities_adapted_norms}
\norm{\exp\bigl(\zeta A_{|\eeemu}\bigr)}_\mutext \le e^{\lambda_{j_0+1,-}\zeta}
\,, \quad\text{and}\quad 
\norm{\exp\bigl(-\zeta A_{|\eeehsu}\bigr)}_\hsu \le e^{-\left(\frac{\cLinMax}{2}-\alpha\right)\zeta} 
\end{equation}
(the ``loss'' $\alpha$ to the right hand side of the second inequality cannot be avoided, due to the existence of at least one Jordan block). For every $U$ in $\rr^{2d}$ let us denote by $U_\hsu$ and $U_\mutext$ the components of $U$ along the supplementary subspaces $\eeehsu$ and $\eeemu$; with symbols, 
\[
U = U_\hsu + U_\mutext 
\quad\text{with}\quad 
U_\hsu \in \eeehsu
\quad\text{and}\quad 
U_\mutext \in \eeemu
\,,
\]
see \cref{fig:invariant_cone}.
\begin{figure}[htbp]
\centering
\includegraphics[width=.4\textwidth]{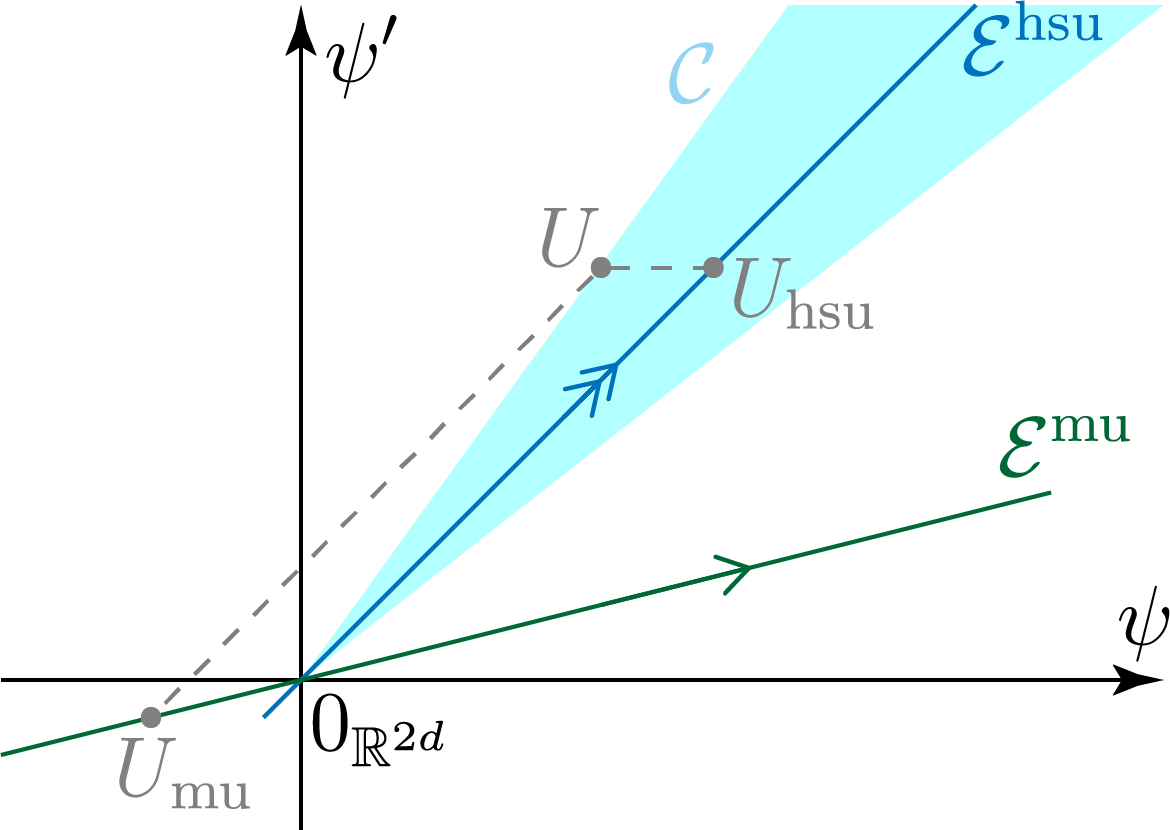}
\caption{Complement subspaces $\eeemu$ and $\eeehsu$ and cone $\ccc$, which is positively invariant under the flow of the linear differential system \cref{linear_differential_system_unperturbed}.}
\label{fig:invariant_cone}
\end{figure}
Let $\beta$ denote a small positive quantity to be chosen below, and let us denote by $\ccc$ the cone around $\eeehsu$ in $\rr^{2d}$, defined as
\begin{equation}
\label{def_ccc}
\ccc = \bigl\{U\in\rr^{2d}: \beta\norm{U_\hsu}_\hsu \ge \norm{U_\mutext}_\mutext\bigr\}
\,, 
\end{equation}
and let $\interior(\ccc)$ denote the interior of $\ccc$. It follows from inequalities \cref{gap_between_mu_and_hsu,inequalities_adapted_norms} that this cone $\ccc$ is positively invariant under the linear flow defined by the matrix $A$ (that, is, by the linear differential system \cref{linear_differential_system_unperturbed}). More precisely, this positive invariance is strict in the sense that, for every $U$ in $\ccc\setminus\{0_{\rr^{2d}}\}$ and for every positive quantity $\zeta$, 
\begin{equation}
\label{cone_invariance_linear}
e^{\zeta A}U \in \interior(\ccc)
\,.
\end{equation}
In addition, if we consider the norm $\norm{\cdot}_{\hsu+\mutext}$ on $\rr^{2d}$ defined as
\[
\norm{U}_{\hsu+\mutext} = \norm{U_\hsu}_\hsu + \norm{U_\mutext}_\mutext
\,,
\]
then, if the positive quantity $\beta$ is small enough, for every $U$ in $\ccc$ and for every positive quantity $\zeta$, 
\begin{equation}
\label{linear_increase_inside_cone}
\norm{e^{\zeta A}U}_{\hsu+\mutext} \ge e^{\left(\frac{\cLinMax}{2} - 2\alpha\right)\zeta}\norm{U}_{\hsu+\mutext}
\,.
\end{equation}
For every nonnegative integer $n$, let $(S^\zeta_n)_\zeta$ denote the flow in $\rr^{2d}$ of the nonlinear system
\begin{equation}
\label{differential_system_psi_cn_V}
\psi'' = c_n \psi' + \nabla V(\psi)
\,.
\end{equation}
\begin{lemma}[robustness of cone invariance properties]
\label{lem:robustness_cone_invariance_properties}
There exists a (small) positive quantity $\rMacro$ such that, for every sufficiently large integer $n$, for every $U$ in $\ccc\setminus\{0_{\rr^{2d}}\}$, and for every positive quantity $\zeta_0$, if the following condition is fulfilled:
\[
\text{for every $\zeta$ in $[0,\zeta_0]$, $S^\zeta_n U$ is defined and belongs to $\widebar{B}_{\rr^{2d}}\bigl(0_{\rr^{2d}},\rMacro\bigr)_{\hsu+\mutext}$,}
\]
then the following conclusions hold: for every $\zeta$ in $[0,\zeta_0]$,
\begin{equation}
\label{properties_nonlinear_flow_inside_cone}
S^\zeta_n U \in\ccc 
\quad\text{and}\quad
\norm{S^\zeta_n U}_{\hsu+\mutext} \ge e^{\left(\frac{\cLinMax}{2} - 3\alpha\right)\zeta} \norm{U}_{\hsu+\mutext}
\,.
\end{equation}
\end{lemma}
In the notation $\rMacro$, the index ``macro'' refers to the fact that this radius will be used to ensure the ``macroscopic'' size of the profiles $\psi_n$, as $n$ goes to $+\infty$. 
\begin{proof}
The strict invariance property \cref{cone_invariance_linear} is robust under a small perturbation, and so is the increase property \cref{linear_increase_inside_cone} (up to replacing $2\alpha$ by $3\alpha$ to the right-hand side of this inequality). The conclusions follow. 
\end{proof}
\subsubsection{Weakly nonlinear behaviour of initial profiles}
Let us pick a positive quantity $\rPert$, large enough so that, for all $u$ and $v$ in $\rr^d$,
\begin{equation}
\label{def_rPert}
\text{if}\quad
(u,v)\in\ccc
\quad\text{and}\quad
\rPert \le \norm{(u,v)}_{\hsu+\mutext} \,,
\quad\text{then}\quad
2 \le \norm{u}
\,;
\end{equation}
this radius will be used to ensure that solutions under consideration are outside of the support of the difference between the initial potential $V$ and the perturbed potentials $W_n$ (thus the index ``pert'' in this notation). It follows from the definition \cref{def_ccc} of $\ccc$ and from the limit \cref{limit_of_tilde_psi_injty_j_prime_over_not_prime} that there exists a positive quantity $\zetaCone$, large enough so that
\begin{equation}
\label{tilde_psi_infty_of_zetaCone_in_C}
\bigl(\tilde{\psi}_{\infty}(\zetaCone),\tilde{\psi}_{\infty}'(\zetaCone)\bigr) \in \interior(\ccc)
\,,
\end{equation}
see \cref{fig:xi_zeta_line}. According to the positive invariance of $\ccc$, $\bigl(\tilde{\psi}_{\infty}(\zeta),\tilde{\psi}_{\infty}'(\zeta)\bigr)$ is still in $\ccc$ for every $\zeta$ greater than $\zetaCone$, and according to the linear increase ensured by inequality \cref{linear_increase_inside_cone}, it may be assumed, up to replacing $\zetaCone$ by a larger positive quantity, that
\begin{equation}
\label{tilde_psi_infty_of_zetaCone_greater_than_rPert}
\rPert < \norm{\bigl(\tilde{\psi}_{\infty}(\zetaCone),\tilde{\psi}_{\infty}'(\zetaCone)\bigr)}_{\hsu+\mutext}
\,.
\end{equation}
It follows from \cref{tilde_psi_infty_of_zetaCone_in_C,tilde_psi_infty_of_zetaCone_greater_than_rPert} that, for every large enough positive integer $n$, 
\begin{equation}
\label{tilde_psi_n_of_zetaCone_in_interior_of_C}
\bigl(\tilde{\psi}_n(\zetaCone),\tilde{\psi}_n'(\zetaCone)\bigr) \in \interior(\ccc)
\quad\text{and}\quad
\rPert < \norm{\bigl(\tilde{\psi}_n(\zetaCone),\tilde{\psi}_n'(\zetaCone)\bigr)}_{\hsu+\mutext}
\,.
\end{equation}
Let us consider again the sequence $(\psi_n)_{n\in\nn}$ of profiles of travelling fronts introduced in \cref{def_psi_n}, more exactly the subsequence defined by the limit \cref{compactness_tilde_psi_n}, and, for every nonnegative integer $n$, let us consider the function $\Psi_n:\rr\to\rr^{2d}$ defined as
\[
\Psi_n(\zeta) = \bigl(\psi_n(\zeta),\psi_n'(\zeta)\bigr)
\,.
\]
\begin{figure}[htbp]
\centering
\includegraphics[width=\textwidth]{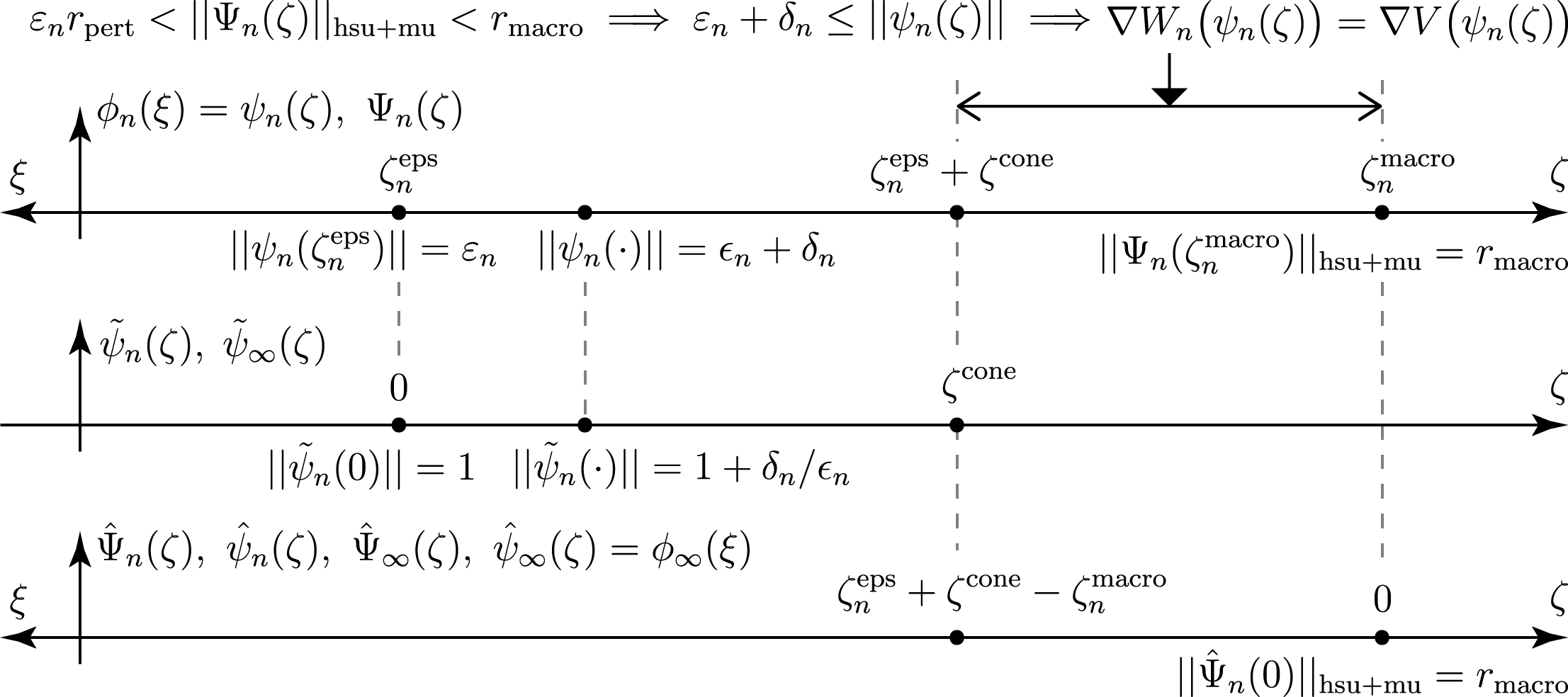}
\caption{Changes of variables (reflections and translations) on the arguments $\xi$ and $\zeta$ of the profiles, notation for the profiles, specific values of the arguments used in the proof, and properties satisfied by the norms of the profiles at these specific values.}
\label{fig:xi_zeta_line}
\end{figure}
\begin{lemma}[weakly nonlinear behaviour of initial profiles]
For every large enough positive integer $n$, there exists a real quantity $\zetaMacro_n$ (see \cref{fig:xi_zeta_line}), greater than $\zetaEps_n+\zetaCone$, such that
\begin{equation}
\label{properties_Psi_n}
\norm{\Psi_n(\zetaMacro_n)}_{\hsu+\mutext} = \rMacro 
\quad
\text{and, for every $\zeta$ in $[\zetaEps_n+\zetaCone,\zetaMacro_n]$,} \quad \Psi_n(\zeta)\in \ccc
\,.
\end{equation}
In addition, $\zetaMacro_n - (\zetaEps_n+\zetaCone) \to +\infty$ as $n\to+\infty$. 
\end{lemma}
\begin{proof}
According to the definition \cref{properties_psi_n_left_of_zero} of $\zetaEps_n$, 
\[
\norm{\psi_n(\zetaEps_n)} = \varepsilon_n\,,
\quad\text{so that}\quad 
\psi_n(\zetaEps_n) \to 0_{\rr^d}
\quad\text{as}\quad
n\to+\infty
\,.
\]
It follows that $\psi_n(\zetaEps_n+\zetaCone)\to 0_{\rr^d}$ as $n\to+\infty$, and thus that
\begin{equation}
\label{Psi_n_of_zetaEps_plus_zetaCone}
\Psi_n(\zetaEps_n+\zetaCone) \to 0_{\rr^{2d}}
\quad\text{as}\quad
n\to+\infty
\,.
\end{equation}
As a consequence, for $n$ large enough, 
\begin{equation}
\label{norm_Psi_n_of_zetaEps_plus_zetaCone_less_than_rho}
\norm{\Psi_n(\zetaEps_n+\zetaCone)}_{\hsu+\mutext} < \rMacro
\,.
\end{equation}
According to the definition \cref{def_tilde_psi_n} of $\tilde{\psi}_n$, 
\[
\Psi_n(\zetaEps_n+\zetaCone) = \varepsilon_n\bigl(\tilde{\psi}_n(\zetaCone),\tilde{\psi}_n'(\zetaCone)\bigr)
\,.
\]
Let us assume that $n$ is large enough so that inequality \cref{norm_Psi_n_of_zetaEps_plus_zetaCone_less_than_rho} holds, and let us consider the set
\begin{align}
\nonumber
\Bigl\{\zeta_0 \in [\zetaEps_n+\zetaCone,+\infty): &\text{ for every $\zeta$ in } [\zetaEps_n+\zetaCone,\zeta_0]\,, \\
\label{properties_defining_set_of_zeta_before_exit_at_rho}
&\Psi_n(\zeta) \in \interior(\ccc)
\quad\text{and}\quad
\varepsilon_n \rPert < \norm{\Psi_n(\zeta)}_{\hsu+\mutext} < \rMacro
\Bigr\}
\,.
\end{align}
It follows from \cref{tilde_psi_n_of_zetaCone_in_interior_of_C,norm_Psi_n_of_zetaEps_plus_zetaCone_less_than_rho} that $\zetaEps_n+\zetaCone$ is in this set (which is therefore nonempty). Let us denote by $\zetaMacro_n$ the supremum (in $\rr\cup\{+\infty\}$) of this set, which, since the conditions defining this set are open, is greater than $\zetaEps_n+\zetaCone$. For every $\zeta$ in $[\zetaEps_n+\zetaCone,\zetaMacro_n)$, it follows from the lower bound in \cref{properties_defining_set_of_zeta_before_exit_at_rho} on $\norm{\Psi_n(\zeta)}_{\hsu+\mutext}$ and from the definition \cref{def_rPert} of $\rPert$ that
\[
2 \varepsilon_n \le \norm{\psi_n(\zeta)}
\,, \quad\text{so that, for $n$ large enough,}\quad
\varepsilon_n + \delta_n \le \norm{\psi_n(\zeta)} 
\,.
\]
This ensures that, on the whole interval $[\zetaEps_n+\zetaCone,\zetaMacro_n)$, $\Psi_n(\cdot)$ is a solution of the differential system \cref{differential_system_psi_cn_V} (for the potential $V$, which does not differ from $W_n$ at the values taken by $\psi_n(\cdot)$ on this interval). In other words, for every $\zeta$ in $[\zetaEps_n+\zetaCone,\zetaMacro_n)$,
\[
\Psi_n(\zeta) = S^{\zeta - (\zetaEps_n+\zetaCone)}_n \Psi_n(\zetaEps_n+\zetaCone)
\,.
\]
It therefore follows from the second among conclusions \cref{properties_nonlinear_flow_inside_cone} of \cref{lem:robustness_cone_invariance_properties} that
\[
\norm{\Psi_n(\zeta)}_{\hsu+\mutext} \ge e^{\left(\frac{\cLinMax}{2} - 3\alpha\right)\bigl(\zeta-(\zetaEps_n+\zetaCone)\bigr)} \norm{\Psi_n(\zetaEps_n+\zetaCone)}_{\hsu+\mutext}
\,.
\]
In view of inequality \cref{gap_between_0_and_hsu}, this shows that $\zetaMacro_n$ is finite and that 
\[
\varepsilon_n \rPert < \norm{\Psi_n(\zetaMacro_n)}_{\hsu+\mutext}
\,,
\]
and it follows from the first among conclusions \cref{properties_nonlinear_flow_inside_cone} of \cref{lem:robustness_cone_invariance_properties} that $\Psi_n(\zetaMacro_n)$ is in the interior of $\ccc$. As a consequence, the following inequality must hold: 
\[
\norm{\Psi_n(\zetaMacro_n)}_{\hsu+\mutext} = \rMacro
\,,
\]
which proves the conclusions \cref{properties_Psi_n}. The last conclusion about the asymptotics of $\zetaMacro_n - (\zetaEps_n+\zetaCone)$ as $n\to+\infty$ follows from the limit \cref{Psi_n_of_zetaEps_plus_zetaCone}. 
\end{proof}
\subsubsection{Obtaining the expected (pulled or pushed) front}
\label{subsubsec:obtaining_expected_front}
Let us consider the functions $\hat{\Psi}_n:\rr\to\rr^{2d}$ and $\hat{\psi}_n:\rr\to\rr^d$ defined as
\[
\hat{\Psi}_n(\zeta) = \Psi(\zetaMacro_n+\zeta)
\quad\text{and}\quad
\hat{\Psi}_n(\zeta) = \bigl(\hat{\psi}_n(\zeta),\hat{\psi}_n'(\zeta)\bigr)
\,,
\]
see \cref{fig:xi_zeta_line}. It follows from the properties \cref{properties_Psi_n} that
\begin{equation}
\label{properties_hat_Psi_n}
\norm{\hat{\Psi}_n(0)}_{\hsu+\mutext} = \rMacro
\quad\text{and, for every $\zeta$ in $[\zetaEps_n+\zetaCone-\zetaMacro_n,0]$}\,,\quad
\hat{\Psi}_n(\zeta)\in \ccc
\,.
\end{equation}
By compactness, up to replacing the sequence $(\hat{\Psi}_n)_{n\in\nn}$ by a subsequence, it may be assumed that there exists a global solution $\hat{\psi}_\infty$ of the differential system
\[
\psi'' = \cLinMax\psi' + \nabla V(\psi)
\,,
\]
such that
\[
\hat{\psi}_n \to \hat{\psi}_\infty 
\quad\text{as}\quad 
n\to+\infty
\,,
\]
uniformly on every compact subset of $\rr$. In addition, if we consider the function $\hat{\Psi}_\infty:\rr\to\rr^{2d}$ defined as
\[
\hat{\Psi}_\infty(\zeta) = \bigl(\hat{\psi}_\infty(\zeta),\hat{\psi}_\infty'(\zeta)\bigr)
\,,
\]
it follows from the properties \cref{properties_hat_Psi_n} that
\begin{equation}
\label{properties_hat_Psi_infty}
\norm{\hat{\Psi}_\infty(0)}_{\hsu+\mutext} = \rMacro
\quad\text{and, for every negative quantity $\zeta$,}\quad
\hat{\Psi}_\infty(\zeta) \in \ccc\,,
\end{equation}
and it follows from the second among the conclusions \cref{properties_nonlinear_flow_inside_cone} of \cref{lem:robustness_cone_invariance_properties} that
\[
\hat{\Psi}_\infty(\zeta)\to 0_{\rr^{2d}}
\quad\text{as}\quad
\zeta\to-\infty
\,.
\]
The function $\phi_\infty:\rr\to\rr^d$ defined as
\[
\phi_\infty(\xi) = \hat{\psi}_\infty(\zeta)
\quad\text{for}\quad
\zeta = -\xi
\,,
\]
is therefore the profile of a wave travelling at the speed $\cLinMax$ and invading the critical point $0_{\rr^d}$. According to uniform bound on the quantity \cref{uniform_bound_profile_front}, this function $\phi_\infty$ is bounded and is therefore the profile of a \emph{front} travelling at the speed $\cLinMax$ and invading the critical point $0_{\rr^d}$. Finally, it follows from the second property of \cref{properties_hat_Psi_infty} that this travelling front must be either pushed or pulled. \Cref{thm:main} is proved. 
\subsubsection*{Acknowledgements} 
The authors are indebted to Thierry Gallay and Romain Joly for their interest and support through numerous fruitful discussions and to Patrick Redont for drawing their attention to the reference \cite{JendoubiPolacik_nonStabSolSemilinHypElliptEquDamp_2003}. 
\printbibliography 
\bigskip
\RamonsSignature
\bigskip

\mySignature
%
%
\end{document}